\documentclass[aop,seceqn,citesort,MSNbibl,dvips]{arximspdf}
\usepackage{mathrsfs}
\usepackage{graphicx}

\doi{10.1214/10-AOP564}
\volume{39}
\issue{2}
\pubyear{2011}
\firstpage{648}
\lastpage{682}

\makeatletter

\newproclaim{remark}{Remark}[section]
\newtheorem{theorem}[remark]{Theorem}
\newtheorem{proposition}[remark]{Proposition}
\newtheorem{lemma}[remark]{Lemma}

\newcommand{\D}{\mathscr{D}}
\newcommand{\Prob}{\mathbb{P}}
\newcommand{\T}{\mathbb{T}}
\newcommand{\E}{\mathbb{E}}
\newcommand{\PP}{\mathbb{P}}
\newcommand{\CP}{\mathcal{P}}
\newcommand{\CC}{\mathcal{C}}
\newcommand{\CB}{\mathcal{B}}
\newcommand{\CO}{\mathcal{O}}
\newcommand{\CL}{\mathcal{L}}
\newcommand{\R}{\mathbb{R}}
\newcommand{\Z}{\mathbb{Z}}
\newcommand{\eps}{\varepsilon}
\newcommand{\F}{\mathscr{F}}
\newcommand{\I}{{\mathscr I}}
\newcommand{\one}{\mathbf1}
\newcommand{\TV}{\mathrm{TV}}

\newcommand{\esssup}{\mathop{\operatorname{ess}\operatorname{sup}}}

\makeatother

\begin{document}
\begin{frontmatter}

\title{Periodic homogenization with an interface: The multi-dimensional case}

\runtitle{Periodic homogenization with an interface}

\begin{aug}
\author[A]{\fnms{Martin} \snm{Hairer}\corref{}\ead[label=e1]{mhairer@cims.nyu.edu}%
\ead[label=e3]{M.Hairer@Warwick.ac.uk}\thanksref{aut1}}
and
\author[B]{\fnms{Charles} \snm{Manson}\ead[label=e2]{charliemanson1982@hotmail.co.uk}}

\thankstext{aut1}{Supported by EPSRC Advanced Research Fellowship Grant EP/D071593/1.}

\runauthor{M. Hairer and C. Manson}

\affiliation{New York University and University of Warwick}

\address[A]{Courant Institute\\
of Mathematical Sciences\\
New York University\\
251 Mercer St.\\
New York, New York 10012\\
USA\\
\printead{e1}\\
\phantom{E-mail: }\printead*{e3}}

\address[B]{Mathematics Institute\\
University of Warwick\\
Coventry CV4 7AL\\
United Kingdom\\
\printead{e2}} 
\end{aug}

\received{\smonth{9} \syear{2009}}
\revised{\smonth{4} \syear{2010}}

%
\begin{abstract}
We consider a diffusion process with coefficients that are periodic
outside of an ``interface region'' of finite thickness. The question
investigated in this article is the limiting long time/large scale
behavior of such a process under diffusive rescaling. It is clear that
outside of the interface, the limiting process must behave like
Brownian motion, with diffusion matrices given by the standard theory
of homogenization. The interesting behavior therefore occurs on the
interface. Our main result is that the limiting process is a
semimartingale whose bounded variation part is proportional to the
local time spent on the interface. The proportionality vector can have
nonzero components parallel to the interface, so that the limiting
diffusion is not necessarily reversible. We also exhibit an explicit
way of identifying its parameters in terms of the coefficients of the
original diffusion.

Similarly to the one-dimensional case, our method of proof relies on
the framework provided by Freidlin and Wentzell [\textit{Ann. Probab.}
\textbf{21} (1993) 2215--2245] for diffusion processes on a graph in
order to identify the generator of the limiting process.
\end{abstract}

%
\begin{keyword}[class=AMS]
\kwd{60H10}
\kwd{60J60}.
\end{keyword}
\begin{keyword}
\kwd{Periodic homogenization}
\kwd{interface}
\kwd{skew Brownian motion}
\kwd{local time}.
\end{keyword}

\end{frontmatter}

\section{Introduction}\label{sec1}

The theory of periodic homogenization is by now extremely well
understood; see, for example, the monographs \cite{lions,MR2382139}.
Recall that
the most basic result states that if $X$ is a diffusion with smooth
periodic coefficients, then the diffusively rescaled process $X^\eps
(t) = \eps X(t/\eps^2)$ converges
in law to a Brownian motion with an explicitly computable diffusion matrix.
If one considers diffusions that are ``locally periodic,'' but with slow
modulations over spatial scales of order~$\eps^{-1}$,
then it was shown in \cite{MR2126985} that the rescaled process
converges in general to some diffusion process with
a computable expression for both its drift and diffusion coefficients.

In this article, we will also consider the ``locally periodic''
situation, but instead of considering slow modulations
of the coefficients, we consider the case of a sharp [i.e., of size
$\CO(1)$] transition between two periodic structures.
In the (much simpler) one-dimensional case, this model was previously
studied in \cite{1}, where we showed that the
rescaled process converges in law to skew Brownian motion with an
explicit expression for the skewness parameter.
In higher dimensions, this model has not yet been studied to the best
of our knowledge. The aim of this article is to
clarify what is the behavior of $X^\eps$ near the interface for very
small values of $\eps$.
It is important to remark at this stage that we do \textit{not} make
the assumption that our diffusion is reversible. As we will see in
Section \ref{sec:mainresult}, there are then situations
in which the limiting process is not reversible either, contrary to the
one-dimensional situation.

One feature of the problem at hand is that
there is no \textit{finite} invariant measure built into the framework
of the problem.
This is unlike most other homogenization problems, even those
exhibiting rather ``bad'' ergodic properties,
such as the random environment case \cite{MR712714,Olla} or the
quenched convergence results for the
Bouchaud trap model \cite{MR2353391}.
Since in our case the invariant measure $\mu$ of $X$ is only $\sigma$-finite,
this leads to two problems when trying to compute the effect of the
behavior of $X$ near the interface
in the limit $\eps\to0$. Indeed, one would ``na{\"\i}vely'' expect
that an effective drift along the interface can be described
by the quantity
%
%
\begin{equation}\label{e:intb}
\int b(x) \mu(dx) .
\end{equation}
One problem with this expression is that there is no obvious natural
normalization for $\mu$. Furthermore, since $b$
is periodic away from the interface and the same is (approximately)
true for $\mu$, this integral certainly does not converge,
even if we consider it as an integral over $\R\times\T^{d-1}$ by
making use of the periodic structure in the directions parallel
to the interface. See however (\ref{e:valuealpha}) and
Proposition \ref
{prop:convalphap} below for the correct way of
interpreting (\ref{e:intb}) and our main result, Theorem \ref
{maintheorem} below, on how this quantity appears
in the construction of the limiting process.

Another common feature of many homogenization results is the usage
of a globally defined corrector function to compensate for the singular
terms appearing in the problem.
This is of course the case for standard periodic homogenization \cite
{lions}, but also for a number of
stochastic homogenization problems, as, for example, in \cite
{MR2031772,MR712714,Olla,MR2475672}.
For the present problem however, it will be convenient to make use of
corrector function that
only cancels the singular terms away from the interface and to treat
the behavior of the limiting
process at the interface by completely different means.

One very recent homogenization result where discontinuous coefficients
appear in the limiting equation
can be found in \cite{MR2480550} (which in turn generalizes \cite
{MR1828773}). However, their framework
is quite different to the one considered here and does not seem to encompass
our problem. Much more closely related problems are homogenization
problems with the presence of a
boundary \cite{MR1696289,Nader}. Those have been mostly studied by
analytical tools so far. In our probabilistic language,
what comes closest to the boundary layers studied in these articles is
the $\sigma$-finite invariant measure of $X$,
which is shown in Proposition \ref{invariantconv} below to converge
exponentially fast to a measure with periodic densities away from
the interface.

For simplicity, we will consider the case of a constant diffusion
matrix, but it is straightforward to adapt the proofs
to cover the case of nonconstant diffusivity as well.
More precisely, we consider the family of processes $X^{\eps}$ taking
values in $\R^{d}$, solutions to the stochastic differential equations
%
%
\begin{equation}\label{firstequation}
dX^{\eps} = \frac{1}{\eps} b \biggl(\frac{X^{\eps}}{\eps}
\biggr)\, ds + dB(s) , \qquad X^\eps(0) = x ,
\end{equation}
where $B$ is a $d$-dimensional standard Wiener process. The drift $b$
is assumed to be smooth and such that $b(x + e_i) = b(x)$ for the unit vectors
$e_i$ with $i=2,\ldots,d$ (but not for $i = 1$). Furthermore, we
assume that there exist smooth vector fields $b_\pm$ with unit period
in \textit{every} direction
and $\eta> 0$ such that
\[
b(x) = b_+(x) , \qquad x_1 > \eta,\qquad b(x) = b_-(x) , \qquad x_1 <
-\eta.
\]
Figure \ref{fig:example} is a typical illustration of the type of
vector fields that we have in mind.

%
%
\begin{figure}

\includegraphics{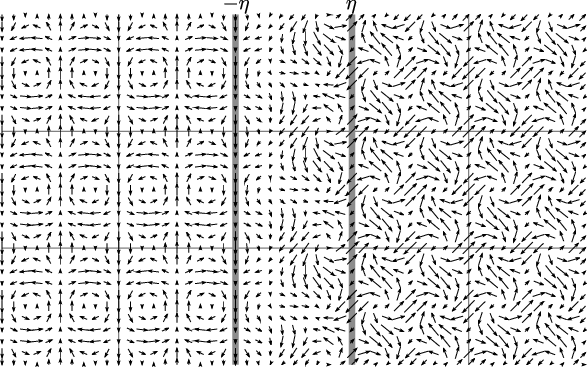}

\caption{Example of a vector field $b$ satisfying our
conditions.}\label{fig:example}
\end{figure}

If we denote by $X$ the same process, but with $\eps= 1$, then the
process $X^\eps$ given by
(\ref{firstequation}) is equal in law to the diffusive rescaling of $X$
by a factor $\frac{1}{\eps}$. In the sequel, we denote the generator
of $X$ by $\CL$
and the generator of $X^\eps$ by $\CL_\eps$. We furthermore denote
by $\CL_\pm$ the generators for the diffusion processes on the torus given
by
%
%
\begin{equation}\label{e:defXpm}
dX^{\pm} = b_\pm(X^{\pm}) \,ds + dB(s) ,
\end{equation}
and by $\mu_\pm$ the corresponding invariant probability measures.
With this notation at hand, we impose the centering condition
$\int_{\T^d} b_\pm(x) \mu_\pm(x) = 0$.

Under these conditions, our main result formulated in Theorem \ref
{maintheorem} below states that the family $X^\eps$ converges in law to
a limiting process $\bar X$. Furthermore, we give an explicit
characterization of $\bar X$, both as the unique solution of a
martingale problem
with some explicitly given generator and as the solution of a
stochastic differential equation involving a local time term on
the interface $\{x_1 = 0\}$. In addition to the homogenized diffusion
coefficients on either side of the interface, this limiting process is
characterized by a ``transmissivity coefficient,'' as well as by a ``drift
vector'' pointing along the interface.

The remainder of this article is structured as follows. After
formulating our main results in Section \ref{sec:mainresult}, we show
tightness of the family
in Section \ref{sec:tight}. In Section~\ref{sec:maintool}, we then
formulate the main tool used in the identification of the limiting
process, namely
a multidimensional analogue of the tool used by Freidlin and Wentzell
in \cite{MR1245308} to study homogenization problems where the
limiting process takes values in a graph. Section \ref{sec:ppm} is
then devoted to the computation of the transmissivity coefficient,
whereas Section \ref{sec:drift} contains the computation of the drift vector.
Finally, we show in Section \ref{sec:wellposed} that the martingale
problem is well-posed and
we identify its solution with the solution to a stochastic differential
equation.

\subsection{Notation}

We define the ``interface'' of width $K$ by
\[
\I_{K} = \{x\in\R^d \dvtx x_1 \in[-K,K]\} .
\]
We also denote by $\partial\I_K$ its boundary.

Frequently throughout the paper we will construct successive escape
and subsequent reentry times particularly when constructing invariant
measures in terms of the invariant measure of an embedded Markov chain
as in \cite{MR0133871}. We will denote such pairs of stopping times as
$\sigma$, $\phi$, which denote escape and reentry times, respectively.
Other stopping times not part of such a sequence will be denoted
by~$\tau$.

\section{The main result}
\label{sec:mainresult}

Before stating the main result, we will first define the various
quantities involved and their relevance.
It is clear that, in view of standard results from periodic
homogenization \cite{lions,MR2382139}, any limiting process
for $X^\eps$ should behave like Brownian motion on either side of the
interface $\I_0 = \{x_1 = 0\}$, with effective diffusion tensors given by
\[
D^\pm_{ij} = \int_{\mathbb{T}^{d}}(\delta_{ik}+\partial_k g_i^\pm
)(\delta_{kj}+\partial_k g_j^\pm)\, d\mu_{\pm} .
\]
(Summation of $k$ is implied.)
Here, the corrector functions $g_\pm\dvtx\T^d \to\R^d$ are the
unique solutions to
$\CL_\pm g_\pm= - b_\pm$ such that
\[
\int_{\T^d} g_\pm(x) \mu_\pm(dx) = 0 .
\]
Since $b_\pm$ are centered with respect to $\mu_\pm$, such functions
do indeed exist.

This justifies the introduction of a differential operator $\bar\CL$
on $\R^d$ defined in two parts by $\bar\CL_{+}$ on $I_+ = \{x_{1}>0\}$
and $\bar\CL_{-}$ on $I_- = \{x_{1}<0\}$ with
%
%
\begin{equation}\label{e:defLpm}
\bar\CL_{\pm} = {D_{ij}^\pm\over2} \partial_i\, \partial_j ,
\end{equation}
then one would expect any limiting process to solve a martingale
problem associated to $\bar\CL$. However, the above definition of
$\bar\CL$ is not complete, since
we did not specify any boundary condition at the interface $\I_0$.

One of the main ingredients in the analysis of the behavior of the
limiting process at the interface is the invariant measure
$\mu$ for the (original, not rescaled) process $X$. It is not clear
a priori that such an
invariant measure exists, since $X$
is not expected to be recurrent in general. However, if we identify
points that differ by integer multiples of $e_j$ for $j = 2,\ldots,d$,
we can interpret $X$ as a process with state space $\R\times\T
^{d-1}$. It then follows from the results in \cite{MR0133871} that this
process admits a $\sigma$-finite invariant measure $\mu$ on $\R
\times\T^{d-1}$.

Note that the invariant measure $\mu$ is \textit{not} finite and can
therefore not be normalized in a canonical way. However, if we define the
``unit cells'' $C_j^\pm$ by
\[
C_j^+ = [j, j + 1] \times\T^{d-1} ,\qquad
C_j^- = [- j - 1, - j] \times\T^{d-1} ,
\]
then it is possible to make sense of the quantity $q_\pm= \lim_{j \to
\infty} \mu(C_j^\pm)$ (we will show in Proposition \ref{invariantconv}
below that this limit actually exists).

Let now $p_{\pm}$ be given by
\[
p_{\pm}=\frac{q_\pm D^{\pm}_{11}}{q_+ D_{11}^{+}+ q_- D_{11}^{-}} ,
\]
which can also we rewritten in a more suggestive way as
%
%
\begin{equation}\label{whichsideitescapes}
{p_+ \over p_-} = {q_+ D_{11}^+ \over q_- D_{11}^-} .
\end{equation}
This is the homogenized diffusion coefficient in the direction
perpendicular to the interface,
weighted by the invariant measure of a unit cell. Comparing with the
one-dimensional case \cite{1}, one would expect this
to yield the likelihood for $X^\eps$ to exit a small (but still much
larger than $\eps$) neighborhood of the interface
on a specific side.
\begin{remark}
The ratio
%
%
\begin{equation}\label{timeonside}
\frac{p_{+}\sqrt{D_{11}^-}}{p_{-}\sqrt{D_{11}^+} + p_{+}\sqrt
{D_{11}^-}}
\end{equation}
gives the asymptotic probability of the process being located in the
rhs ($+$) of the interface after a long time.
This follows from the weak convergence of the first component to a skew
Brownian motion with (possibly) different diffusion coefficients on
either side of the interface.\vspace*{1pt} If we rescale this skew BM
on either side
of the interface by $\sqrt{D_{11}^{\pm}}$ to obtain a standard
skew BM, we can use the scale function of BM to finish the verification
of (\ref{timeonside}).
\end{remark}

However, unlike in the one-dimensional case, these quantities are not
yet sufficient to characterize the limiting process.
The reason is that since $X^\eps$ is expected to spend time
proportional to $\eps$ in the interface, but the drift
is of order $\eps^{-1}$ there, it is not impossible that the limiting
process picks up a nontrivial drift along the interface.
It turns out that this drift can be described by the coefficients
$\alpha_{j}$ given by
%
%
\begin{equation}\label{e:valuealpha}
\alpha_j = 2 \biggl({p_+ \over D_{11}^+} + {p_- \over D_{11}^-} \biggr)
\int_{\R\times\T^{d-1}} \bigl(b_j(x) + \CL g_j(x)\bigr) \mu(dx) ,
\end{equation}
where $\mu$ is again normalized in such a way that $q_+ + q_- = 1$ and
where $g$ is any smooth function agreeing with $g_\pm$
on either side of the interface (see Section \ref{sec:tight}).
\begin{remark}\label{rem:renorm}
Since $\int_{\R\times\T^{d-1}} \CL\phi(x) \mu(dx) = 0$ for
every smooth \textit{compactly supported} function $\phi$,
one should interpret the integral on the right-hand side of (\ref
{e:valuealpha}) as a ``renormalized'' form of the
intuitive more meaningful quantity (\ref{e:intb}).
\end{remark}
\begin{remark}
The expression (\ref{e:valuealpha}) is useful in order to generate
examples with nonvanishing values
for the coefficients $\alpha_i$.
\end{remark}

Given all of these ingredients, we can construct an operator $\bar\CL
$ as follows. The domain
$\D(\bar\CL)$ of $\bar\CL$ consists of functions $f \dvtx\R^d
\to\R$ such that:
\begin{itemize}
\item The restrictions of $f$ to $I_+$, $I_-$ and $\I_0$ are smooth.
\item The partial derivatives $\partial_i f$ are continuous for $i
\ge2$.
\item The partial derivative $\partial_1 f(x)$ has right and left
limits $\partial_{1} f|_{I_{\pm}}$ as $x \to\I_0$ and these limits
satisfy the gluing condition
%
%
\begin{equation} \label{derivative}
p_{+}\,\partial_{1} f|_{I_{+}}-p_{-}\partial_{1} f|_{I_{-}}+\sum
_{j=2}^{d}\alpha_{j}\,\partial_{j} f=0 .
\end{equation}
\end{itemize}
For any $f \in\D(\bar\CL)$, we then set $\bar\CL f(x) = \CL_\pm
f(x)$ for $x \in I_\pm$.
With these definitions at hand, we can state the main result of the article.
\begin{theorem}\label{maintheorem}
The family of processes $X^{\eps}$ converges in law to the unique
solution $\bar X$
to the martingale problem given by the operator $\bar{\CL}$.
Furthermore, there exist matrices
$M_\pm$ and a vector $K \in\R^d$ such that this solution solves the SDE
%
%
\begin{equation}\label{e:limit}
d\bar X(t) = \one_{\bar X_1 \le0} M_- \,dW(t) + \one_{\bar X_1 > 0}
M_+ \,dW(t) + K \,dL(t) .
\end{equation}
where $L$ denotes the symmetric local time of $\bar X_1$ at the origin
and $W$ is a standard $d$-dimensional
Wiener process. The matrices $M_\pm$ and the vector $K$ satisfy
\[
M_\pm M_\pm^T = D^\pm,\qquad K_1 = p_+ - p_- ,\qquad K_j = \alpha
_j ,
\]
for $j =\{2,\ldots,d\}$.
\end{theorem}

In Figure \ref{fig:simulation}, we show an example of a numerical
simulation of the process studied in this article.
The figure on the left shows the small-scale structure (the periodic
structure of the drift is drawn as a grid). One can
clearly see the periodic structure of the sample path, especially to
the left of the interface. One can also see that the effective
diffusivity is not necessarily proportional to the identity. In this
case, to the left of the interface,
the process diffuses much more easily horizontally
than vertically.

%
%
\begin{figure}

\includegraphics{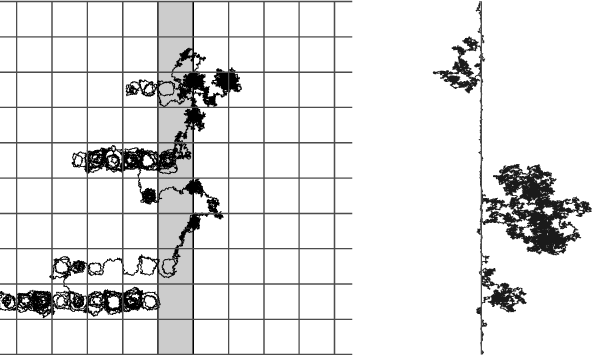}

\caption{Sample paths at small (left) and large (right) scales.}
\label{fig:simulation}
\end{figure}

The picture to the right
shows a simulation of the process at a much larger scale. We used a
slightly different vector field for the drift in order
to obtain a simulation that shows clearly the strong drift experienced
by the process when it hits the interface.
\begin{remark}
Since the quadratic variation of $\bar X$ has a discontinuity at \mbox
{$\bar
X_1 = 0$}, we do have to specify
which kind of local time $L$ is. Using the symmetric local time yields
nicer expressions.
See, for example, \cite{MR1083357,MR2280299} for a definition of the
symmetric local time.
\end{remark}

Analyzing what this means for a simple example, we consider the case of
a two-dimensional problem where we have $b_{1}=0$ and $b_{2}=f(x_{1})$
for $f$ a smooth function that is zero outside of $\I_\eta$. Clearly,
$p_{\pm}=\frac{1}{2}$. In this case, the invariant measure $\mu$ of
the process $X$ is given by ${1\over2}$ times
Lebesgue measure on $\R\times S^1$ and we can choose $g=0$. This
implies that we then simply have
\[
\alpha_{2}=\int_{\mathbb{R}}f(x) \,dx ,
\]
as one would expect.

\section{Tightness of the family}
\label{sec:tight}

The aim of this section is to prove the following tightness result.
\begin{theorem}\label{theo:tight}
Denote by $\Prob^\eps$ the law of $X^\eps_x$ on $\CC(\R_+,\R^d)$.
Then the family $\{\Prob^\eps\}_{\eps\in(0,1]}$ is tight.
\end{theorem}

Similar to what happens in the classical theory of periodic
homogenization, it will be very convenient to construct a ``corrected
process''
$Y$, obtained by adding to $X$ a corrector function that cancels out to
first order the effect of the small oscillations.
To this aim, we introduce a smooth function $g\dvtx\R^d \to\R^d$
which is periodic in the directions
$2,\ldots,d$ and such that $g(x) = g_+(x)$ for $x_1 \ge\eta$ and
similarly for $x_1 \le-\eta$. (Recall that $g_\pm$ was defined
in Section \ref{sec:mainresult}.)
We do not specify the behavior of $g$ inside the interface $\I_\eta$,
except that it has to be smooth in the whole space and
periodic in the directions parallel to the interface. We fix such a
function $g$ once and for all from now on.
We furthermore denote by $Y^{\eps}$ the process defined by $Y^{\eps}
= X^{\eps} + \eps g(\eps^{-1}X^{\eps})$, as well as
$y = x + \eps g(x/\eps)$ for its initial condition.

Defining the corrected drift $\tilde b(x) = (\CL g + b)(x)$
and the corrected diffusion coefficient
$\tilde\sigma_{ij}(x) = \delta_{ij} + \partial_j g_i(x)$, it
follows from It\^o's formula that the $i$th component of
$Y^\eps_y$ satisfies
%
%
\begin{equation}\label{e:defY}
(Y^\eps_y)_i(t) = y_i + \int_{0}^{t} \frac{1}{\eps}
\tilde b_i \biggl({1\over\eps}X^{\eps}_{x}(s) \biggr) \,ds + \int
_{0}^{t}\tilde\sigma_{ij} \biggl({1\over\eps}X^{\eps}_{x}(s)
\biggr) \,dW_{j}(s) .
\end{equation}
It is very important to note that the corrected drift $\tilde b$
\textit{vanishes}
outside of $\I_\eta$, so that the process $Y$ is subject to a large
drift only when $X$ is inside the interface.

Our main tool in the proof of Theorem \ref{theo:tight} is the
following result, which is very similar to
\cite{MR532498}, Theorem 1.4.6.
\begin{proposition}\label{Theorem 1}
Let $\mathscr{P}$ be a family of probability measures on $\Omega= \CC
(\R_+,\R^d)$ and denote by $x$ the canonical process on $\Omega$.
Assume that
\[
\lim_{R\nearrow\infty}\sup_{\Prob\in\mathscr{P}}\Prob\bigl(
|x(0)|\geq R \bigr) = 0 .
\]
Furthermore, for any given $\rho>0$, let $\tau_{0}=0$, and define
recursively $\tau_{i+1} = \inf_{t>\tau_{i}}|x(t)-x(\tau_{i})|>\rho
$. Assume that the limit
%
%
\begin{equation}\label{lemma 1.4.4'}
\lim_{\delta\rightarrow0} \esssup\Prob[\tau_{n+1}-\tau_{n}
\leq
\delta|\mathscr{F}_{\tau_{n}}] \rightarrow0,\qquad\mbox{$\Prob$ a.s.,
on $\{\tau_{n}<\infty\}$}
\end{equation}
holds uniformly for every $\Prob\in\mathscr P$ and every $n\ge0$.
Then the family of probability measures $\mathscr{P}$ is tight on
$\Omega$.
\end{proposition}
\begin{pf}
The proof is similar to that of Theorem 1.4.6 in \cite{MR532498},
except that their Lemma 1.4.4 is replaced by (\ref{lemma 1.4.4'}).

Fix an arbitrary final time $T>0$. Furthermore, denote for $\omega\in
\Omega$
\[
N_\rho= N_\rho(\omega) = \min\{n\dvtx\tau_{n+1}>T\},
\]
and the modulus of continuity by $\delta_{\rho}$,
\[
\delta_{\rho} = \delta_{\rho}(\omega) = \min\{\tau_{n} - \tau
_{n-1}\dvtx1\leq n \leq N_\rho(\omega)\} .
\]
Note that this expression depends on $\rho$ via the definition of the
stopping times $\tau_i$.

With this notation at hand, tightness follows as in \cite{MR532498} if
one can show that $\lim_{\delta\to0} \sup_{\Prob\in\mathscr
{P}}\Prob(\delta_\rho\leq\delta) = 0$
for every fixed $\rho> 0$. As in \cite{MR532498}, one has for every
$k>0$ the bound
\[
\Prob(\delta_\rho\leq\delta) \leq
\sum_{i=1}^{k}\mathbb{E}\bigl[ \Prob[\tau_{i+1}-\tau_{i} \leq
\delta| \mathscr{F}_{\tau_{i}}] \bigr]+ \Prob(N_\rho>k) .
\]
For every fixed $k>0$, the first term then converges uniformly to $0$
by assumption.
Since the second term is independent of $\delta$, it remains to verify
that converges to~$0$ as
$k\to\infty$, uniformly over $\mathscr{P}$ (convergence for every
fixed $\PP\in\mathscr{P}$ is trivial but not sufficient for our needs).

This is a consequence of \cite{MR532498}, Lemma 1.4.5, provided that
one can find $\lambda< 1$ such that
$\E[e^{-(\tau_{i+1}-\tau_{i})}| \mathscr{F}_{\tau_{i}}
] \leq\lambda$.
This in turn follows from
\begin{eqnarray*}
\E\bigl[e^{-(\tau_{i+1}-\tau_{i})}| \mathscr{F}_{\tau_{i}}\bigr]
&\leq& \Prob[\tau_{i+1}-\tau_{i}\leq
t_{0}|\mathscr{F}_{\tau_{i}}]+e^{-t_{0}}\Prob[\tau_{i+1}-\tau
_{i}>t_{0}| \mathscr{F}_{\tau_{i}}]
\\
&\leq& e^{-t_{0}} + (1-e^{-t_{0}})\Prob[\tau_{i+1}-\tau_{i}\leq
t_{0}| \mathscr{F}_{\tau_{i}}] .
\end{eqnarray*}
Indeed, by choosing $t_{0}$ sufficiently small, this term can be made
strictly less than~$1$, provided
that $\Prob[\tau_{i+1}-\tau_{i}\leq t_{0}| \mathscr{F}_{\tau
_{i}}]$ tends to zero uniformly
(over the members of $\mathscr P$ and over $i$) as $t_{0}$ tends to
zero, which is precisely our assumption.
\end{pf}

We now turn to the
following.
\begin{pf*}{Proof of Theorem \ref{theo:tight}}
Recall that we defined the process $Y^\eps= X^\eps+ \eps g(\eps^{-1}
X^\eps)$ in Section \ref{sec:mainresult}.
Note then that, just as in \cite{1}, Proposition 2.5, the tightness of
the laws of $X^{\eps}_{x}$ is equivalent to that of the laws of
$Y^{\eps}_{x}$.
Therefore, all that remains to be shown is that we have the bound (\ref
{lemma 1.4.4'}) for the law of $Y^\eps$,
uniformly over $\eps\in(0,1]$. The approach that we use is to
consider separately the martingale part and the bounded variation part
for $Y^{\eps}_{y}$ given by (\ref{e:defY}),
and to show that the probability of either of these moving by at least
$\frac{\rho}{2}$ during a time interval $\delta$ tends to zero
uniformly over the initial condition.

Given any fixed $\rho, \gamma>0$, we want to show that there exists a
sufficiently small $\delta>0$ such that $\Prob(\tau_{n+1}-\tau
_{n}\leq\delta| \mathscr{F}_{\tau_{n}})<\gamma$ uniformly over
$\Prob\in\mathscr{P}$ (i.e., uniformly over the laws of $Y^\eps
_x$ with $\eps\in(0,1]$) and $n$. We split the contributions from the
martingale and the bounded variation parts in the following way:
%
%
\begin{eqnarray}\label{10}
&&
\Prob(\tau_{n+1}-\tau_{n}\leq\delta| \mathscr{F}_{\tau_{n}})\nonumber\\
&&\qquad= \Prob_{X(\tau_{n})} \Bigl( \sup_{t<\delta}|Y(t)-Y(0)|>\rho
\Bigr)\nonumber\\
&&\qquad\leq\sup_x \Prob_x \biggl( \sup_{t<\delta} \biggl|\frac{1}{\eps
}\int_{0}^{t}\tilde b_{i}(\eps^{-1}X^{\eps}_{x}(s)) \,ds
\biggr|>\frac{\rho}{2} \biggr)\nonumber\\[-8pt]\\[-8pt]
&&\qquad\quad{} + \sup_x \Prob_x \biggl( \sup_{t<\delta} \biggl|\int
_{0}^{t}\tilde\sigma_{ij}(\eps^{-1}X^{\eps}_{x}(s))\,
dW_{j}(s) \biggr|>\frac{\rho}{2} \biggr) \nonumber\\
&&\qquad\le{2\over\eps\rho} \sup_x \E_x \int_{0}^{t}|\tilde
b_i(\eps^{-1}X^{\eps}_{x}(s))| \,ds\nonumber\\
&&\qquad\quad{} + {2\over\rho}\sup_x \E_x \sup_{t \le\delta} \biggl|\int
_{0}^{t}\tilde\sigma_{ij}(\eps^{-1}X^{\eps}_{x}(s))
\,dW_{j}(s) \biggr| .\nonumber
\end{eqnarray}
Here, we used the Chebychev's inequality to obtain the last bound.
Since the functions $\tilde\sigma_{ij}$ are uniformly bounded, the
stochastic integral appearing in the
second term is easily bounded by $\CO(\sqrt\delta)$ by the
Burkholder--Davis--Gundy inequalities.
Furthermore, by the definition of the corrector function $g$, there
exists $\tilde\eta>0$ such that $\tilde b(x) = 0$ for $x \notin\I
_{\tilde\eta\eps}$, so that
there exists a constant $C$ such that
%
%
\begin{equation}\label{e:boundtight}\quad
\Prob(\tau_{n+1}-\tau_{n}\leq\delta| \mathscr{F}_{\tau_{n}})
\le{C\over\rho\eps} \sup_x \E_x \biggl(\int_0^\delta\one_{\I
_{\tilde\eta\eps}}(X^\eps_x(s)) \,ds \biggr)
+ {C \sqrt\delta\over\rho} .
\end{equation}
For fixed $\rho> 0$, the second term obviously goes to $0$ as $\delta
\to0$, uniformly in $\eps$, so it remains to consider the first term.
As one would expect from the expression for the local time of a
Brownian motion, it turns out that the expected
time spent by the process in $\I_{\tilde\eta\eps}$ scales like
$\eps\sqrt{\delta}$, thus showing that this term is also of order
$\sqrt\delta/ \rho$. Once we are able to show this, the proof is complete.

The occupation time of the interface appearing in the first term of
(\ref{e:boundtight}) is bounded by the trivial
estimate $C\delta/ (\rho\eps)$, which goes to $0$ as $\delta
\rightarrow0$ provided that we consider $\eps\ge\sqrt\delta$, say.
We can therefore
assume without any loss of generality in the sequel that we consider
$\eps< \sqrt\delta$.

The idea to bound the occupation time is the following. We decompose
the trajectory for the process $X^\eps$ into excursions away from the
interface,
separated by pieces of trajectory inside the interface. We first show
that if the process starts inside the interface, then the expected time
spent in the interface before making a new excursion is of order $\eps
^2$. Then, we show that each excursion has a probability
at least $\eps/ \sqrt\delta$ of being of length $\delta$ or more.
This shows that in the time interval $\delta$ of interest,
the process will perform at most of the order of $\sqrt\delta/ \eps$
excursions, so that the total time spent in the interface is indeed
of the order $\eps\sqrt\delta$, thus showing that the first term in
(\ref{e:boundtight}) behaves like $\sqrt\delta/ \rho$, as expected.

More precisely, we first choose two constants $K > 0$ and $\hat K > 0$
such that the chain of implications
%
%
\begin{equation}\label{e:chainimplications}
\{X^\eps\in\I_{\tilde\eta\eps}\} \Rightarrow\{Y^\eps\in\I
_{\hat K \eps}\} \Rightarrow\bigl\{X^\eps\in\I_{(K-1) \eps}\bigr\}
\Rightarrow\{X^\eps\in\I_{K\eps}\}
\end{equation}
holds. We then set up a sequence of stopping times in the following
way. We set $\phi_0 = 0$ and we set recursively
\begin{eqnarray*}
\sigma_n &=& \inf\{t \ge\phi_n \dvtx X^\eps(t) \notin\I_{K\eps
}\} ,\\
\phi_n &=& \inf\{t \ge\sigma_{n-1} \dvtx Y^\eps(t) \in\I_{\hat
K\eps}\} .
\end{eqnarray*}
[Note that we can have $\sigma_0 = 0$ if the initial condition does
not belong to $\I_{K\eps}$. Apart from that, the second implication
in (\ref{e:chainimplications}) shows that increments from
one stopping time to the next are always strictly positive.] This
construction was chosen in such a way that the times when
$X^\eps\in\I_{\tilde\eta\eps}$ always fall between $\phi_n$ and
$\sigma_n$ for some $n \ge0$.
In particular, if we set
\[
N = \inf\{n \ge0 \dvtx\phi_{n+1} - \sigma_n \ge\delta\} ,
\]
then we have the bound
\begin{eqnarray*}
\sup_x \E_x \biggl(\int_0^\delta\one_{\I_{\tilde\eta\eps
}}(X^\eps_x(s)) \,ds \biggr)
&\le&\sup_x \E_x \Biggl(\sum_{n=0}^N (\sigma_n - \phi_n) \Biggr) \\
&=& \sup_x \sum_{n=0}^\infty\E_x \bigl((\sigma_n - \phi_n) \one
_{N\ge n} \bigr)\\
&=& \sum_{n=0}^\infty\sup_x\PP_{x}(N \ge n) \sup_{x}\E_{x}
\bigl(\E
_{X^\eps(\phi_n)}
\sigma_1\bigr),
\end{eqnarray*}
where we used the strong Markov property and the fact that $\{N \ge n\}
$ is $\F_{\phi_n}$-measurable in order to obtain the
last identity. It follows from the definition of $N$ that this
expression is in turn bounded by
\[
\sup_{x\in\R^d} \E_x \sigma_0 \sum_{n \ge0}\Bigl(\sup_{x \notin
\I_{K\eps}} \PP_x(\phi_0 < \delta)\Bigr)^n = {\sup_{x\in\R^d}
\E_x
\sigma_0 \sup_{x\notin\I_{K\eps}} \PP_x(\phi_0 < \delta)
\over
\inf_{x\notin\I_{K\eps}} \PP_x(\phi_0 \ge\delta)} .
\]
We now bound both terms appearing in this expression separately.

First, we turn to the expected escape time from the interface, $\E_x
\sigma_0$.
The idea is to use a comparison argument just like in \cite{1},
Proposition 3.8. We define a ``worst-case scenario'' process
$V_{x}^{\varepsilon}$, which is the solution to the SDE with initial
condition $x$, diffusion coefficient 1 and drift coefficient given by
$b^{\varepsilon}_{V}$, where
\[
b^{\varepsilon}_{V}(x) = \cases{
\dfrac{-b_{V}}{\eps}, &\quad for $x \ge0$,\vspace*{2pt}\cr
\dfrac{b_{V}}{\varepsilon}, &\quad for $x < 0$,}
\]
for some constant $b_V > 0$.
We then have the following lemma.
\begin{lemma}\label{lem:boundwc}
There exist $b_V>0$ and $\tilde K > 0$ such that, if we define $\tau
^{\tilde K} = \inf\{t \ge0 \dvtx V_x^\eps(t) \notin\I_{\tilde
K\eps}\}$, we have
\[
\E_x \sigma_0 \le\E_x \tau^{\tilde K} ,
\]
for every $x \in\R^d$.
\end{lemma}

The proof of Lemma \ref{lem:boundwc} is almost identical to that of
\cite{1}, Proposition 3.8, so we are going to omit it. A
straightforward calculation
using the particular form of the drift coefficient for $V$ allows to
check that there exists indeed a constant $C>0$ such that the bound
\[
\sup_x \E_x \tau^{\tilde K} \le C \eps^2 ,
\]
holds so that, combining this with Lemma \ref{lem:boundwc}, we have
$\sup_{x\in\R^d} \E_x \sigma_0 \le C\eps^2$.

Let us now turn to the bound on $\PP_x(\phi_0 \ge\delta)$. The idea
here is to look at the process $Y^\eps$ instead of $X^\eps$ and to time-change
it in such a way that we can compare it to a standard Brownian motion.
Note first that the last two implications in (\ref{e:chainimplications})
show that if we start with $X^\eps$
anywhere outside of $\I_{K\eps}$, then the first component of $Y^\eps
$ has to travel by at least $\eps$ before the process $Y^\eps$ can
hit $\I_{\hat K\eps}$.
Furthermore, it follows from (\ref{e:defY}) that the time change $C_t$
such that $Y^\eps(C_t)$ is a standard Brownian motion satisfies
$C_t \ge ct$ for some $c > 0$. It therefore follows that, setting $H(z)
= \inf_{t>0}\{B_{t}>z\}$, one has the lower bound
\[
\inf_{x \notin\I_{K\eps}} \PP_x(\phi_0 \ge\delta) \ge\PP
\bigl(H(\eps) \ge\delta/c\bigr) .
\]
The explicit expression for the law of $H(z)$ given in \cite
{MR1477407}, page 163, equation 2.02, yields in turn
\[
\Prob\bigl(H(\eps) \ge\delta/c\bigr) =\int_{\delta/(c\eps^2)}^{\infty
}\frac{e^{-1/(2t)}}{\sqrt{2\pi} t^{{3/2}}} \,dt.
\]
It follows immediately that this in turn is bounded from below by $C
\eps/\sqrt{\delta}$ for some $C>0$, provided that
$\eps\le\sqrt\delta$. Collecting these bounds completes the proof
of Theorem \ref{theo:tight}.
\end{pf*}

\section{Main tool for identifying the limit process}
\label{sec:maintool}


Instead of considering a graph as before, we will consider a
generalized multidimensional version different from that considered by
Freidlin and Wentzell in \cite{MR1245308}, Section 6. Note that the
generalization considered
here is different (and actually simpler) than the one considered in
\cite{MR2253881}.
We consider processes in $\mathbb{R}^{d}$ and we set $I_{-} = \{x \in
\R^d \dvtx x_1 < 0\}$, and similarly for
$I_+$. We consider a family of $\R^d$-valued processes $X^\eps$ and
we denote by $\tau^{\eps}$
the first hitting time of $\I_{\eps\eta}$. Correspondingly, $\tau
^{\delta}$ is the first escape time of the
set $\I_{\delta}$ by $X^{\eps}$.

With this the main tool, will be the following multidimensional
analogue of~\cite{MR1245308}, Theorem 4.1.
\begin{theorem}\label{maintheoremm}
Let $\bar{\CL}_{i}$ be second order differential operators on $I_{i}$
with bounded coefficients and let $D_i$ be
some sets of test functions over $I_i$ whose members are bounded and
have bounded
derivatives of all orders.
Suppose that for $i \in\{+,-\}$, any function $f \in D_{i}$ and for
any $\lambda> 0$, the bound
%
%
\begin{eqnarray} \label{2.1}
&&\mathbb{E}_{x} \biggl[e^{-\lambda\tau^{\varepsilon}}f
(X^{\varepsilon}(\tau^{\varepsilon})) - f(X^{\eps
}(0))\nonumber\\[-8pt]\\[-8pt]
&&\qquad\hspace*{0pt}{} + \int_{0}^{\tau^{\varepsilon}}e^{-\lambda t}
\bigl(\lambda f
(X^{\varepsilon}(t)) - \bar{\CL}_{i}f(X^{\varepsilon
}(t)) \bigr) \,dt \biggr]
= O(k(\varepsilon)) ,\nonumber
\end{eqnarray}
holds as $\varepsilon\rightarrow0$, uniformly with respect to $x \in
I_{i}$. Assume furthermore that the rate $k$ is such that $\lim
_{\varepsilon\rightarrow0}k(\varepsilon) = 0$.

Assume that, for every $\lambda>0$ and every $i \in\{+,-\}$, there
exist functions $u_{i,\lambda} \in D_i$
such that $\bar{\CL}_{i}u_{i,\lambda}(x) = \lambda u_{i,\lambda
}(x)$ holds for $x \in I_i$ with $|x_1| \le1$
and such that $u_{\pm,\lambda}(x) = 1$ for $x_1 = 0$ and $x_1 = \pm1$.

Assume that there exists a rate $\delta= \delta(\varepsilon)
\rightarrow0$
such that $\delta(\varepsilon)/k(\varepsilon) \rightarrow\infty$
as $\varepsilon\rightarrow0$ and such that for $\lambda> 0$,
%
%
\begin{equation} \label{2.2}
\mathbb{E}_{x}^{\varepsilon} \biggl[\int_{0}^{\infty}e^{-\lambda t}
\one_{(-\delta,\delta)}(X_{1}^{\varepsilon}(t)) \,dt
\biggr]\rightarrow0
\end{equation}
as $\varepsilon\rightarrow0$, uniformly in the initial point. Assume
the convergence
%
%
\begin{equation} \label{2.3}
\mathbb{P}_{x}^{\varepsilon}[X^{\varepsilon}(\tau^{\delta
}) \in I_{i}] \rightarrow p_{i} ,
\end{equation}
holds uniformly in $x$ in the set $\I_{\eps\eta}$ for some constants
$p_\pm$ with $p_+ + p_- = 1$. Assume
furthermore that there exist constants $\alpha_j$ and $C$ such that
%
%
\begin{equation} \label{othercomponents}
\frac{1}{\delta} \mathbb{E}_{x}^{\eps} [X_{j}^{\varepsilon
}(\tau^{\delta})-x_j ]\rightarrow
\alpha_{j} ,\qquad
\frac{1}{\delta^{2}} \mathbb{E}_{x}^{\eps} \bigl[
\bigl(X_{j}^{\varepsilon}(\tau^{\delta})-x_j\bigr)^{2}
\bigr]\leq
C ,
\end{equation}
for $j \ge2$. Again, the limit is assumed to be uniform over $x \in\I
_{\eps\eta}$ as $\eps\rightarrow0$, and the inequality is assumed
to be uniform over all $\eps\in(0,1]$ and all $x \in\I_{\eps\eta}$.

Let then $D$ be the set of continuous functions $f \dvtx\R^d \to\R
$ such that the restriction of $f$ to $I_i$ belongs
to $D_i$ and such that
the gluing condition (\ref{derivative}) holds. Then, for any fixed
$f\in D$, $t_{0} \geq0$ and $\lambda> 0$,
%
%
\begin{eqnarray} \label{2.4}
&&\Delta(\varepsilon) = \esssup \biggl| \mathbb
{E}_{x}^{\varepsilon}
\biggl[\int_{t_{0}}^{\infty}e^{-\lambda t} [ \lambda f
(X^{\varepsilon}(t)) - \bar{\CL}f(X^{\varepsilon
}(t)) ] \,dt \nonumber\\[-8pt]\\[-8pt]
&&\qquad\hspace*{120pt}{}- e^{-\lambda t_{0}}f(X^{\varepsilon}(t_{0}))
\Big|
\mathscr{F}_{[ 0, t_{0}]} \biggr] \biggr| \rightarrow0\nonumber
\end{eqnarray}
as $\varepsilon\rightarrow0$, uniformly with respect to $x$. In
particular, every weak limit of $X^\eps$
as $\eps\to0$ satisfies the martingale problem for $\bar\CL$.
\end{theorem}
\begin{remark}
Note that we did not specify how ``large'' the sets $D_i$ of admissible
test functions need to be. If these sets
are too small, then the theorem still holds, but the corresponding
martingale problem might become
ill-posed.
\end{remark}
\begin{pf*}{Proof of Theorem \ref{maintheoremm}}
Since the proof is virtually identical to that of~\cite{MR1245308},
Theorem 4.1, we only sketch it here.
The basic idea behind the proof given by Freidlin and Wentzell is to
rewrite (\ref{2.4}) using the strong Markov property of $X^{\eps}$
as a sum of terms between successive stopping times. To this effect,
set, for example, $\sigma_{0} = 0$ and then
recursively $\phi_{n} = \inf\{t>\sigma_{n}\dvtx X_{1}^{\eps}(t) \in
\I_{\eps\eta}\}$, $\sigma_{n+1} = \inf\{t>\phi_{n}\dvtx X_{1}^{\eps
}(t) \notin\I_\delta\}$.
They then break up the term produced from (\ref{2.4}) into two sums of
analogous terms between times $\sigma_{n}$ and $\phi_{n}$ and those
between $\phi_{n}$ and $\sigma_{n+1}$.

The terms covering the time intervals $[\sigma_n, \phi_n]$ are
bounded exactly as in \cite{MR1245308}, making use of (\ref{2.1}), together
with the bound $\sum_{n}\mathbb{E}_{x}e^{-\lambda\sigma_{n}} = \CO
(1/\delta)$ which follows from the existence of the functions
$u_{i,\lambda}$ just as in \cite{MR1245308}.

Using assumption (\ref{2.2}), the terms covering the time intervals
$[\phi_n, \sigma_{n+1}]$ are then simplified to
\[
\sum_{n}e^{-\lambda\phi_{n}} \bigl(f(X^{\eps}(\sigma
_{n+1}))-f(X^{\eps}(\phi_{n})) \bigr) ,
\]
modulo contributions that converge to $0$ as $\eps\to0$.
Since the expectation of this term is bounded by
\[
\sup_{x \in\I_{\eta\eps}} \E_x \bigl(f(X^{\eps}(\tau
^\delta))-f(x)\bigr) \sum_{n} \E e^{-\lambda\phi_{n}} ,
\]
and since we already know that $\sum_n \E e^{-\lambda\phi_{n}} = \CO
(1/\delta)$, in remains to show that the supremum is
of order $o(\delta)$. It follows from Taylor's expansion and the fact
that $f \in\CC^2$ outside of the interface,
that on the event $\Omega_+ \stackrel{\mathrm{def}}{=}\{X_1^\eps
(\tau^\delta) > 0\}$,
one has
\[
f(X^{\eps}(\tau^\delta))-f(x) = \delta\partial_1
f(x)|_{I_+} + \sum_{i=2}^d \partial_i f(x) \bigl(X_i^{\eps}(\tau
^\delta) - x_i\bigr)
+ \CO\bigl(| X_i^{\eps}(\tau^\delta) - x_i|^2\bigr) ,
\]
and similarly on $\Omega_- = \{X_1^\eps(\tau^\delta) < 0\}$.
Combining this with (\ref{othercomponents}), we thus have
\begin{eqnarray*}
\E_x \bigl(f(X^{\eps}(\tau^\delta))-f(x)\bigr)
&=& \delta\,\partial_1 f(x)|_{I_+}\PP_x(\Omega_+) + \delta\,\partial
_1 f(x)|_{I_-}\PP_x(\Omega_-) \\
&&{} + \delta\sum_{i=2}^d \alpha_i \,\partial_i f(x) +
o(\delta) .
\end{eqnarray*}
Since we assume that $\PP_x(\Omega_\pm)\to p_\pm$ uniformly over
$x\in\I_{\eta\eps}$, the required bound now follows from the
gluing condition.
\end{pf*}

Most of the remainder of this article is devoted to the verification of
the assumptions of Theorem \ref{maintheoremm}.
The bounds (\ref{2.1}) and (\ref{2.2}) will be relatively
straightforward to verify and this will form the content of
the remainder of this section. The convergence (\ref{2.3}) is the one
that is most difficult to obtain and will be
the content of Section \ref{sec:ppm}. Finally, we will show that (\ref
{othercomponents}) holds in Section \ref{sec:drift}. We start by the
following result.
\begin{lemma}
Let $\bar\CL_\pm$ be as in (\ref{e:defLpm}) and let $X^\eps$ be the
family of processes from Section \ref{sec:mainresult}. Then, the bound
(\ref{2.1}) holds with $k(\eps) = \eps$ for every $\lambda> 0$ and
for every smooth bounded function
$f \dvtx I_i \to\R$ that has bounded derivatives of all orders.
\end{lemma}
\begin{pf}
It follows from \cite{1}, Lemma 3.4, that, for any initial point $x$
with $x_{1} \neq0$ and for $\eps$ sufficiently small so that
$x \notin\I_{\eps\eta}$,
%
%
\begin{equation}\label{e:average}
\E_{x} \biggl[\int_{0}^{\tau^{\eps}} e^{-\lambda s}f
(X^{\varepsilon}(s))h \biggl(\frac{X^{\varepsilon
}(s)}{\varepsilon} \biggr) \,ds \biggr] = \CO(\varepsilon) ,
\end{equation}
for $h$ centered with respect to $\mu_+$ (resp., $\mu_-$ if $x_1 <
0$). We assume that $x_1 > 0$ from now on, but the calculations are identical
for the case $x_1 < 0$.

Note now that it suffices to obtain the bound (\ref{2.1}) for the
family of processes $Y^\eps$, since
$\|Y^\eps(t) - X^\eps(t)\| = \CO(\eps)$, uniformly. Applying It\^
o's formula to
$e^{-\lambda\tau^{\eps}}\times f(Y^{\varepsilon}(\tau^{\eps}))$,
we obtain the identity
\begin{eqnarray*}
e^{-\lambda\tau^{\eps}}f(Y^{\varepsilon}(\tau^{\eps})
) &=&
f(y) + \int_{0}^{\tau^{\eps}}-\lambda e^{-\lambda
s}f(Y^{\varepsilon}(s)) \,ds\\
&&{} + {1\over
2}\int_{0}^{\tau^{\eps}}e^{-\lambda s}(\tilde\sigma
_{ik}\tilde
\sigma_{kj}) \biggl(\frac{X^{\eps}}{\eps} \biggr)\,\partial_{ij}^2
f(Y^{\eps}(s)) \,ds \\
&&{} + \int_{0}^{\tau^{\eps}}e^{-\lambda s}
\tilde
\sigma_{ik} \biggl(\frac{X^{\eps}(s)}{\eps} \biggr)\,\partial
_{i}f(Y^{\eps}(s)) \,dW_{k}(s) .
\end{eqnarray*}
Since $|Y^\eps- X^\eps| \le\CO(\eps)$ and since all derivatives of
$f$ are assumed to be bounded, it then follows from (\ref{e:average}) that
\begin{eqnarray*}
\E(e^{-\lambda
\tau^{\eps}}f(Y^{\varepsilon}(\tau^{\eps})) ) &=& f(y)
-\lambda\E\int_{0}^{\tau^{\eps}} e^{-\lambda
s}f(Y^{\varepsilon}(s)) \,ds\\
&&{} + {1\over2}\E
\int_{0}^{\tau^{\eps}}e^{-\lambda s}D_{ij}^+ \,\partial_{ij}^2
f(Y^{\eps}(s)) \,ds + \CO(\eps) ,
\end{eqnarray*}
which is precisely the required result.
\end{pf}

Additionally we have that the solution to $\bar\CL_{i}u = \lambda u$
on $I_{i}$, $u=1$ on $\{x_1 = 0\}$ and $\{x_1 = \pm1\}$, is bounded
and has bounded derivatives of all orders. This follows from the fact
that $u$ is given explicitly by $u(x) = C_1e^{\sqrt{\lambda
(D_{11}^{\pm})^{-1}}x_1} + C_2 e^{-\sqrt{\lambda(D_{11}^{\pm})^{-1}}x_1}$
for some constants $C_i$.
We now show that the process $Y^\eps$ satisfies the bound (\ref{2.2}),
that is, it does not spend too much time in the
vicinity of the interface.
\begin{lemma}\label{lem:timeinterface}
If we choose $\delta= \eps^\alpha$ for any $\alpha\in({1\over2},
1)$, then (\ref{2.2}) holds
for the family of processes $X^\eps$ from Section \ref{sec:mainresult}.
\end{lemma}
\begin{pf}
Again, it suffices to show the bound for the process $Y^\eps$ since it
differs from $X^\eps$ by $\CO(\eps)$.
We would like to use an argument similar to what can be used in the
one-dimensional case \cite{1}, that is, we
time-change the corrected process $Y^\eps$ in such a way that it
becomes a diffusion with
diffusion coefficient $1$. Its drift then vanishes outside of the
interface and is
bounded by $K/\eps$ for some $K>0$. At this stage, one compares this
process to the ``worst-case scenario'' process
$Z^\eps$ given by
\[
dZ^\eps= \hat b(Z^\eps) \,dt + dB(t) ,
\]
where the drift $\hat b$ is given by
\[
\hat b(z) =
\cases{
-K\eps^{-1}, &\quad if $z \in[0,l\eps)$, \cr
K\eps^{-1}, &\quad if $z \in(-l\eps,0)$, \cr
0, &\quad otherwise,}
\]
for some $l\in\mathbb{R}$. It can then be shown that $Z^\eps$ spends
more time in the interface than $Y^\eps$ does, so that the requested
bound can be obtained from a simple calculation.

The problem with this argument is that in the multi-dimensional case
the time-change required to turn the
first component of $Y^\eps$ into a diffusion with unit diffusion
coefficient is given by
%
%
\begin{equation}\label{e:timechange}
T_t = \inf\Biggl\{s\in\mathbb{R}_{+}\dvtx\int_{0}^{s}\sum
_{i=1}^{n} \bigl(\delta_{1i}+\,\partial_{i}g_{1}(\eps
^{-1}X^{\eps}(u)) \bigr)^{2}\,du>t \Biggr\} .
\end{equation}
We do not know of an argument giving a uniform bound \textit{from
below} on the quantity appearing under
the integral in this expression. Therefore, an upper bound on the time
spent by the process $Z^\eps$ in
the interval
$(-\delta, \delta)$ does not give us any control on the time spent by
$Y^\eps$ (and therefore $X^\eps$)
in that interval.

Because of this, we modify our argument in the following way. We break
up the integral in (\ref{2.2}) as
%
%
\begin{eqnarray}\qquad
\label{A3}
\mathbb{E}_{x} \biggl[\int_{0}^{\infty}e^{-\lambda t}\one_{(-\delta
, \delta)}(Y_{1}^{\eps}(t)) \,dt \biggr]
&=& \mathbb{E}_{x} \biggl[\int_{0}^{\infty}e^{-\lambda t}\one
_{(-c\eps, c\eps)}(Y_{1}^{\eps}(t)) \,dt \biggr]
\\
\label{A2}
&&{} + \mathbb{E}_{x} \biggl[\int_{0}^{\infty}e^{-\lambda t}\one
_{(-\delta,-c\eps)}(Y_{1}^{\eps}(t)) \,dt \biggr]
\\
&&{} + \mathbb{E}_{x} \biggl[\int_{0}^{\infty}e^{-\lambda
t}\one_{(c\eps, \delta)}(Y_{1}^{\eps}(t)) \,dt
\biggr] ,\nonumber
\end{eqnarray}
where $Y_{1}^{\eps}$ is the first component of $Y^{\eps}$ and $c$ is
a value to be determined. By symmetry, the last two terms are of the
same order, so that it is sufficient to bound the first two terms.
In order to bound the first term, we use the argument outlined above,
but we replace $Y^\eps$ by the
process $\tilde Y^\eps$ given by $\tilde Y^\eps(t) = X^\eps(t) +
\eps\tilde g(\eps^{-1} X^\eps(t))$, where the
corrector $\tilde g$ has the following properties:
\begin{enumerate}
\item The function $\tilde g(x)$ is smooth, periodic in the
variables parallel to the interface, and equal to $g(x)$ for $x
\notin\I_{c_1}$ for some $c_1$.
\item One has the implication $Y^\eps\in\I_{c\eps} \Rightarrow
\tilde Y^\eps\in\I_{c_2\eps}$ for some $c_2 < c_1$.
\item If $\tilde Y^\eps\in\I_{c_2\eps}$, then $\tilde g(\eps
^{-1}X^\eps) =0$.
\end{enumerate}
It is always possible to satisfy these properties by choosing $c_1$
sufficiently large and setting $g=0$
in a sufficiently wide band around the interface. We now set $\tilde
Z(t) = \tilde Y(\tilde T_t)$, where $\tilde T_t$ is defined as in (\ref
{e:timechange}), but with $g$ replaced by $\tilde g$, so that it follows
from the second property
that one has the bound
\begin{eqnarray*}
\mathbb{E}_{x}\int_{0}^{\infty}e^{-\lambda t}\one_{(-c\eps, c\eps
)}(Y_{1}^{\eps}(t)) \,dt
&\le&\mathbb{E}_{x}\int_{0}^{\infty}e^{-\lambda t}\one_{(-c_2\eps
, c_2\eps)}(\tilde Y_{1}^{\eps}(t)) \,dt \\
&\le&\mathbb{E}_{x}\int_{0}^{\infty}e^{-\lambda T_t}\one_{(-c_2\eps
, c_2\eps)}(\tilde Z_1^{\eps}(t)) \,dT_t .
\end{eqnarray*}
At this stage, we remark that since the function $\tilde g$ has bounded
derivatives, there exists a constant
$K_1$ such that $T_t \ge K_1 t$ almost surely. On the other hand, it
follows from the last property that one actually has $dT_t = \,dt$
whenever $\tilde Y^\eps\in\I_{c_2 \eps}$, so that this expression
is bounded by
\[
\mathbb{E}_{x} \int_{0}^{\infty}e^{-K_{1}t}\one_{(-c_2\eps,
c_2\eps)}(\tilde Z_1^{\eps}(t)) \,dt .
\]
This expression in turn can be bounded by $\CO(\eps)$ just as in
\cite{1}.

We now proceed to bounding the term (\ref{A2}). For this, let us first
introduce a constant $c_3 < c$
and make $c$ from (\ref{A3}) sufficiently large such that:
\begin{enumerate}[4.]
\item[4.] The implication $X^{\eps}(t) \in\I_{c_3 \eps}
\Rightarrow Y^{\eps}(t) \in\I_{c \eps}$ holds.
\item[5.] One has $c_3 > \eta+1$.
\end{enumerate}

Then, we define a series of stopping times $\{\phi'_{n}\}_{n}$ and $\{
\sigma'_{n}\}_{n}$ recursively
by $\phi'_{-1} = 0, \ldots, \sigma'_{n} = \inf\{t\geq\phi
'_{n-1}\dvtx X_{1}^{\eps}(t) \notin(-2\delta,-c_3\eps+\eps)\}$ and
$\phi'_{n} = \inf\{t \ge\sigma'_{n}\dvtx X^{\eps}_{1}(t)\in(-\delta
,-c_3\eps)\}$.

Now we can use the strong Markov property as in \cite{MR1245308},
Lemma 4.1, with the stopping times $\phi'_{n}$ to obtain the bound
$\mathbb{E}_{x}[ \sum_{n=0}^{\infty}e^{-\lambda\sigma
'_{n}(\eps)}] = O(\frac{1}{\eps})$, uniformly in
the initial point $x$ for $x\in\{x\dvtx
x_{1}=-c_3\eps+\eps\}\cup\{x\dvtx
x_{1}=-2\delta\}$. This is a consequence of the fact that $\mathbb
{E}_{x}[e^{-\lambda\sigma'_{0}}] = 1-O(\eps)$ uniformly.
Furthermore, it follows from the definition of these stopping times,
property 4 and the strong Markov property that (\ref{A2}) is bounded by
%
%
\begin{eqnarray}\label{A5}\qquad
\mathbb{E}_{x} \int_0^\infty e^{-\lambda t} \one_{(-\delta,-c_3\eps
)}(X_{1}^{\eps}(t)) \,dt
&\le&\mathbb{E}_{x} \sum_{n\ge0}\int_{\phi'_{n-1}}^{\sigma'_{n}}
e^{-\lambda t} \,dt \nonumber\\
&\le&\lambda^{-1}\mathbb{E}_{x}\sum_{n\ge0}e^{-\lambda\phi
'_{n-1}}(\sigma'_n - \phi'_{n-1})\nonumber\\[-8pt]\\[-8pt]
&\le&\lambda^{-1} \Biggl(\mathbb{E}_{x}\sum_{n=0}^{\infty
}e^{-\lambda\phi'_{n}(\eps)} \Biggr) \sup_x \E_x \sigma'_0\nonumber\\
&\le&{C\over\eps\lambda} \sup_x \E_x \sigma'_0 .\nonumber
\end{eqnarray}
It follows that it suffices to be able to choose $\delta$ in such a
way that $\E_x \sigma'_0$
is $o(\eps)$ uniformly in the initial point.
Specifically, we will show that (\ref{A5}) is $\CO(\delta^{2})$, so
that the claim follows.

This will be a consequence of the following result.
\begin{lemma}\label{lem:boundexit}
Let $X^-$ be as in (\ref{e:defXpm}) and define $X^{-,\eps}(t) = \eps
X^-(\eps^{-2}t)$.
Let $\tau= \inf\{t>0 \dvtx X^{-,\eps}_1(t) \notin[-1,0]\}$. Then,
there exists a constant $C$ such that
\[
\E_x \tau\le C ,
\]
independently of $\eps\in(0,1]$ and independently of $x \in\R^d$.
\end{lemma}

Before we prove Lemma \ref{lem:boundexit}, we use it to complete the
proof of Lemma \ref{lem:timeinterface}.
It follows from property 5. that up to time $\sigma'_{0}$, the process
$X^\eps$ is identical in law
to the process $X^{-,\eps}$. Furthermore, the stopping time $\sigma
'_0$ is certainly bounded
from above by the first exit time
of the first component of $X^{-,\eps}$ from $(-2\delta,0)$. Rescaling
space by a factor $2\delta$ and rescaling time correspondingly by
$4\delta^2$, we deduce from Lem\-ma~\ref{lem:boundexit} that $\E\sigma
'_0 \le4C\delta^2$, uniformly
in the initial condition as required.
\end{pf}

We now turn to the proof of the lemma.
\begin{pf*}{Proof of Lemma \ref{lem:boundexit}}
Denote by $U$ the region $\{x \in\R^d \dvtx x_1 \in[-1,0]\}$ and
define $f^\eps$ by $f^\eps(x) = \E_x \tau$. Then
$f^\eps$ satisfies
\[
\CL^{\eps}f^{\eps}=-1,\qquad
f^{\eps}(x) = 0 \qquad\mbox{for $x\in\partial U$,}
\]
where $\CL^\eps= {1\over2}\Delta+ \eps^{-1} b_-(\eps^{-1}\cdot
)\nabla_x$.
In order to obtain a bound on $f$, we will give a uniformly bounded
(uniformly over $\eps$) function $g^{\eps}$ such that it satisfies
%
%
\begin{equation} \label{geps}
\CL^{\eps}g^{\eps} = -1,\qquad g^{\eps}(x)\geq0 \qquad\mbox{for $x\in
\partial U$.}
\end{equation}
It then follows from the strong maximum principle (which we can apply
since our diffusion is periodic
in the directions in which $U$ is unbounded) that $g^\eps\ge f^\eps$,
so that the requested bound holds.


We use a standard multiscale expansion for $g^\eps$ of the form
\[
g^\eps= g_0 + \eps g_1 + \eps^2 g_2 .
\]
Now to find such a $g^{\eps}$. We proceed by starting off with a
constant order term, that is, the typical term one would expect for the
escape time if we were dealing with a Brownian motion, then removing
the order $\frac{1}{\eps}$ terms that arise when the operator
$L^{\eps}$ acts on the constant order term by adding an order $\eps$
term. Then finally we add an order $\eps^{2}$ term to remove the
constant order terms that are produced by the action of $L^{\eps}$ on
the order $\eps$ term. Incidentally, this approach of correction works
exactly with the maximum order term in $\eps$ being 2 and produces a
series of terms that are known and have the right properties to provide
a uniform bound.

Taking guidance from the fact that the homogenized process is given by
Brownian motion, we make the ansatz
$g_0(x) = C_2 - C_1 x_1 (1+x_1)$, for $C_1$ and $C_2$ two constants to
be determined.
Applying $\CL^{\eps}$ to $g_0$ yields
\[
\CL^{\eps}g_0(x) = -C_{1}-\frac{C_1}{\eps}b_{-,1} \biggl(\frac
{x}{\eps} \biggr)(1+2x_1)
\]
for $b_{-,1}$ the first component of $b_{-}$. Our aim now is to choose
$g_1$ in such a way that $\CL g_1$ contains a term of order $\eps
^{-1}$ that
precisely cancels out the second term in this expression. Denote as in
the introduction by $g_-$ the unique
centered solution to the Poisson equation
%
%
\begin{equation}\label{e:Poisson}
\CL g_{-}=b_{-} ,
\end{equation}
where $\CL= {1\over2}\Delta+ b_-\nabla_x$ is the generator for the
nonrescaled process.
We then set $g_1(x) = C_1 (1+2x_1) g_{-,1}(\eps^{-1}x)$, where
$g_{-,1}$ is the first component of $g_{-}$, and we note that
%
%
\begin{eqnarray}\label{A12}
\eps\CL^{\eps}g_1(x) &=& \frac{C_1}{\eps} b_{-,1} \biggl(\frac
{x}{\eps} \biggr)(1+2x_{1})\nonumber\\
&&{}
+ 2C_{1}b_{-,1} \biggl(\frac{x}{\eps} \biggr)g_{-,1} \biggl(\frac
{x}{\eps} \biggr) + 2C_{1}\frac{\partial g_{-,1}}{\partial x_{1}}
\biggl(\frac{x}{\eps} \biggr)\\
& = &\frac{C_1}{\eps} b_{-,1} \biggl(\frac{x}{\eps} \biggr)(1+2x_{1})
+ C_1 F\biggl(\frac{x}{\eps}\biggr) ,\nonumber
\end{eqnarray}
for some periodic function $F$ independent of $\eps$ and of $C_1$.
The term involving $F$ appearing in this expression is still of order
one, so we aim to compensate it
by a judicious choice of $g_2$. It is not necessarily centred with
respect to the invariant measure
$\mu$ of our process, but there exists a periodic centred function $h$
such that
\begin{eqnarray*}
\CL h &=& F - K ,\\
K &=& \int F(x) \mu(dx) = - \int|\nabla
g_{-,1}(x)|^2 \mu(dx) + 2\int{\partial g_{-,1} \over\partial
x_1} \mu(dx) .
\end{eqnarray*}
Finally, setting $g_2(x) = -h(\eps^{-1}x)$, we obtain
%
%
\begin{equation}\label{A13}
\CL^{\eps}g^{\eps} = C_{1}(K-1) = -C_1 \int|e_1 - \nabla
g_{-,1}(x)|^2 \mu(dx) .
\end{equation}
Since the integral is strictly positive, the right-hand side can be
made to be equal to $-1$.
Furthermore, since the corrector terms $\eps g_1 + \eps^2 g_2$ are
uniformly bounded for $\eps< 1$,
it is straightforward to find a constant $C_2$ that ensures that $g(x)
\ge0$ for $x \in\partial U$, thus concluding
the proof.
\end{pf*}
%

\section{Computation of the transmissivity coefficient}
\label{sec:ppm}

The aim of this section is to prove that
the following proposition holds.
\begin{proposition}\label{probabilityconverge}
The identity (\ref{2.3}) holds for the family of processes $X^\eps$
in Section \ref{sec:mainresult} with
$p_\pm$ given by (\ref{whichsideitescapes}).
\end{proposition}

Let us first introduce some notation. Given a starting point $x \in\I
_\eta$, we set
$p_+^{x,k} = \PP_x (X(\tau^{(k)}) > 0)$, and similarly for
$p_-^{x,k}$, where $\tau^{(k)}$ is the first hitting time of
$\partial\I_k$. We furthermore set
\[
\bar p_+^k = \sup_{x \in\I_\eta}p_+^{x,k} ,\qquad \underline p_+^k
= \inf_{x \in\I_\eta}p_+^{x,k} ,\qquad
p_+^{(k)} = {1\over2} (\bar p_+^k + \underline p_+^k) ,
\]
and similarly for $p_-$. It is clear that Proposition \ref{probabilityconverge}
follows if we can show that
$p_+^k$ converges to a limit satisfying (\ref{whichsideitescapes}) and
$\bar p_+^k - \underline p_+^k \to0$ as $k \to\infty$.

We will first show the latter, as it is relatively straightforward to
show. In order to show the convergence
of $p_+^k$, our main ingredient will be to show that the invariant
measure $\mu(dx)$ for the process $X$
looks more and more similar to $\mu_\pm(dx)$ as $x_1 \to\pm\infty
$. Note that in this whole section,
we will always consider $X$ and $X_{\pm}$ as processes on $\R\times
\T^{d-1}$, obtained by identifying points
$(x,y)$ such that $x_1 = y_1$ and $x_j - y_j \in\Z$ for $j \ge2$.
With this interpretation, the interface is compact
and we will show that the processes are recurrent. If we were to
consider them as processes in $\R^d$,
they would \textit{not} be recurrent for $d \ge3$.

Before we show that indeed $\bar p_+^k - \underline p_+^k \to0$, we
obtain some recurrence properties of $X$ and ensure that it visits any
open set in
$\I_\eta$ sufficiently often before the hitting time $\tau^{(k)}$.
%
%
\begin{lemma}\label{lemgamma}
Fix a neighborhood $\gamma\subset\I_{\eta}$. Then the probability
for $X$ to enter
$\gamma$ before hitting $\partial\I_k$, starting from an arbitrary
initial point in $\I_{\eta}$ tends to
$1$ uniformly as $k \rightarrow\infty$. In particular, the process
$X$ is recurrent.
\end{lemma}

Our first step in showing this result is to argue that if the process
starts at distance $\CO(1)$ of the interface,
then it will return to the interface with overwhelming probability
before exiting $\I_k$.
\begin{lemma}\label{lemgam}
There exists\vspace*{1pt} $K>0$ such that the probability, starting at $x$, for $X$
to return to $\I_{\eta}$
before hitting $\partial\I_k$, is bounded from above by $1 - {x-K
\over k}$ and from below by $1 - {x + K\over k}$.
\end{lemma}
\begin{pf}
Denote by $f^k(x)$ the probability of hitting $\I_{\eta}$ before
$\partial\I_k$, starting from $x$.
We assume without loss of generality that $x_1 > 0$, since the case
\mbox{$x_1 < 0$} follows using the same argument.
The function $f^k$ then satisfies the equation $\CL f^k = 0$, endowed
with the boundary conditions
$f^k(x) = 1$ if $x_1 = \eta$ and $f^{k}(x) = 0$ if $x_1 = k$. As in
the proof of Lemma \ref{lem:boundexit},
we aim to construct a function $g^k$ satisfying $\CL g^k = 0$ and such
that either $g^k (x) \le f^k(x)$
on the two boundaries or $g^k (x) \ge f^k(x)$ on the two boundaries.
The claim then follows from the maximum principle.

Let $g_+$ be as in (\ref{e:Poisson}) and set
\[
g^k(x) = 1 - k^{-1}\bigl(K + x_1 - g_{+,1}(x)\bigr),
\]
for some constant $K$ to be determined.
It is straightforward to check that $g^k$ does indeed satisfy $\CL g^k
= 0$,
as well as the required inequalities on the boundary, provided that $K$
is either sufficiently large
or sufficiently small.
This concludes the proof.
\end{pf}

We now use the result of Lemma \ref{lemgam} to prove Lemma \ref
{lemgamma}. This is done using the
strong Markov property in conjunction with success/failure trials.
\begin{pf*}{Proof of Lemma \ref{lemgamma}}
Consider the two hyperplanes that delimit $\I_{\eta}$ and two further
hyperplanes at distance $m$ from $\I_{\eta}$, with $m$ a sufficiently
large constant to be determined later. We then break the process into
excursions from $\partial\I_{\eta}$
to $\partial\I_{\eta+ m}$ and back.

More precisely, we define two sets of stopping times $\{\sigma
^{m}_{n}\}_{n}$ and $\{\phi^{m}_{n}\}_{n}$ recursively by $\sigma
^{m}_1 = \inf\{t\geq0\dvtx X(t)\in\partial\I_{\eta+m}\}, \ldots
, \phi^{m}_{n}=\inf\{t>\sigma^{m}_{n}\dvtx X(t)\in\I_\eta\}$,
$\sigma^{m}_{n+1} = \inf\{t>\phi^{m}_{n}\dvtx X(t)\in\partial\I
_{\eta+m}\}$. We furthermore denote by $\F_n$ the $\sigma$-algebra
generated by trajectories of $X$ up to the time $\phi^m_n$ and by
$\bar\F_n$ the $\sigma$-algebra generated by trajectories of $X$ up
to the time $\sigma^m_{n+1}$.
We also denote by $\tau_\gamma$ the first hitting time of the set
$\gamma$ and by $\tau^{(k)}$ the first hitting time of the set
$\partial\I_k$.

It follows from the ellipticity of $X$ and the resulting smoothness of
its transition
probabilities that there exists some $p>0$ such that $\inf_{x \in
\partial\I_\eta} P_1(x,\gamma) = 2p > 0$. Furthermore, it is
straightforward, for instance using a comparison argument with a
process with constant drift away from the interface and using the
continuity of paths, to show that
%
%
\begin{equation}\label{e:probhit}
\lim_{m \to\infty} \sup_{x \in\I_\eta} \PP_x(\sigma^{m}_1 \le
1) = 0 .
\end{equation}
It follows that we can choose $m$ large enough so that the probability
appearing in (\ref{e:probhit}) is bounded
above by $p$. As a consequence, for such a choice of $m$, one has the
almost sure bound
%
%
\begin{equation}\label{e:boundX1}
\PP(\tau_\gamma< \sigma^m_{n+1} | \F_n) \ge p .
\end{equation}
On the other hand, it follows from Lemma \ref{lemgam} that the
probability that the process hits
$\partial\I_k$ between $\sigma^{m}_n$ and $\phi^{m}_n$ is bounded
from \textit{above} uniformly
by $\beta_k = \CO(k^{-1})$ so that, almost surely,
%
%
\begin{equation}\label{e:boundX2}
\PP\bigl(\tau^{(k)} < \phi^m_{n+1} | \bar\F_n\bigr) \le\beta
_k .
\end{equation}
Note furthermore that by construction the event appearing in (\ref
{e:boundX1}) is $\bar\F_n$-measurable.

Denote now by $Y_n$ a Markov chain with states $\{-1,0,1\}$ such that
$\{\pm1\}$ are absorbing and
such that $P(Y_{n+1} = -1 | Y_n = 0) = p$, $P(Y_{n+1} = 1 | Y_n =
0) = \beta_k$. As a consequence of
(\ref{e:boundX1}) and (\ref{e:boundX2}), it is then possible
to couple $Y$ and $X$ in such a way that the following two implications
hold almost surely:
\begin{eqnarray*}
\{Y_n = 0 \mbox{ and } Y_{n+1} = -1\} &\Rightarrow& \bigl\{\phi^m_n <
\tau_\gamma< \sigma^m_{n+1}<\tau^{(k)}\bigr\}, \\
\bigl\{\sigma^m_{n+1} < \tau^{(k)} < \phi^m_{n+1}<\tau_{\gamma}\bigr\}
&\Rightarrow& \{(Y_n = 0 \mbox{ and } Y_{n+1} = 1)\}.
\end{eqnarray*}
It follows that the probability of entering $\gamma$
before the hitting time $\tau^{(k)}$ is bounded from below by
\[
\PP\bigl(\tau_\gamma< \tau^{(k)}\bigr) \ge\PP\Bigl(\lim_{n \to\infty} Y_n = -1\Bigr)
= {p \over p+\beta_k} .
\]
Since $p$ is fixed and $\beta_k = \CO(k^{-1})$, this quantity can be
made arbitrarily close to $1$.

This shows that the set $\gamma$ is recurrent for $X$. Since
furthermore $X$ has transition probabilities that have strictly positive
densities with respect to Lebesgue measure (as a consequence of the
ellipticity of the equations describing it), recurrence follows
from \cite{MR1287609}, Theorem 8.0.1.
%
\end{pf*}

We now use this result to prove
the following proposition.
\begin{proposition}\label{converges} $\bar p_{+}^{k}-\underline p_{+}^{k}
\rightarrow0$ as $k\to\infty$.
\end{proposition}
\begin{pf}
The idea is to use the fact that, before the process
exits $\I_k$, it has had sufficient amount of time to forget about its
initial condition by visiting a small set on which a strong minorizing
condition holds for its transition probabilities.

Fix a value $\beta> 0$. Our aim is to show that there then exists $k_0
> 0$ such that
\[
\underline p_\pm^{k} \ge p_\pm^{0,k} - \beta,
\]
say, for every $k \ge k_0$. Since $p_+^{x,k} = 1-p_-^{x,k}$, the claim
then follows.
We restrict ourselves to the bound for $p_+$ since the other bound can
be obtained in exactly the same way.

The argument is now the following. It follows from the smoothness of
transition probabilities that there exists
a neighborhood $\gamma$ of the origin such that the transition
probabilities at time $1$ for $X$,
starting from $\gamma$ satisfy the
lower bound
\[
\rho(y) = \inf_{x \in\gamma} P_1(x,y) ,
\]
with $\int_{\R\times\T^{d-1}} \rho(y) \,dy \ge1-\beta/2$. It then
follows immediately that
for $x \in\gamma$, one has $p_+^{x,k} \ge p_+^{0,k} - \beta/2 - \PP
_x(\exists t \le1 \dvtx X(t) \in\partial\I_k)$. For
arbitrary $x$, it therefore follows from the strong Markov property that
\[
p_+^{x,k} \ge p_+^{0,k} - \beta/2 - \sup_{y \in\gamma} \PP_y
\bigl(\exists t \le1 \dvtx X(t) \in\I_k\bigr) -
\PP_x (\mbox{$X$ hits $\partial\I_k$ before $\gamma$}
) .
\]
The last term can be made smaller than $\beta/ 4$ by Lemma \ref
{lemgamma}. The remaining term
$\PP_y(\exists t \le1 \dvtx X(t) \in\I_k)$ on the other
hand was already shown to be arbitrary
small in (\ref{e:probhit}).
\end{pf}

We next show that the invariant measure of the process converges to
that of the relevant periodic process with increasing distance from the
interface.
\begin{proposition}\label{invariantconv}
Let $A$ denote a bounded measurable set and denote by $\mu$ the
(unique up to scaling)
invariant $\sigma$-finite measure of the process $X$. Denote
furthermore by $\mu_{\pm}$ the invariant measure of the relevant
periodic process, normalized in such a way that $\mu_\pm
([k,k+1]\times\T^{d-1}) = 1$ for
every $k \in\Z$. Then there exist normalization constants $q_\pm$
such that
%
%
\begin{equation}\label{e:convA}
\lim_{k \to\infty} \bigl(|\mu(A + k)-q_+\mu_{+}(A)| + |\mu(A -
k)-q_-\mu_{-}(A)|\bigr) =0 .
\end{equation}
(Here $k$ is an integer.) Furthermore, this convergence is exponential,
and uniform over the set $A$ if we restrict its diameter.
\end{proposition}
\begin{remark}
We used the shorthand notation $A+k$ for $\{x+k \dvtx x\in A\}$.
\end{remark}
\begin{pf*}{Proof of Proposition \ref{invariantconv}}
We restrict ourselves to the estimate of $\mu(A+k)$, since the one on
$\mu(A-k)$ is similar. For fixed $k \ge0$, we introduce the sequence
of stopping times given by $\phi^{(k)}_0 = \inf\{t \ge0 \dvtx X_1(t)
= k\}$ and then recursively
$\sigma^{(k)}_n = \inf\{t \ge\phi^{(k)}_n \dvtx |X_1(t) - k| = 1\}$,
$\phi^{(k)}_{n+1} = \inf\{t \ge\sigma^{(k)}_n \dvtx X_1(t) = k\}$.
This allows us to define an embedded Markov chain $Z^{(k)}$ on $\T
^{d-1}$ by setting
$Z^{(k)}_n = \Pi X(\phi^{(k)}_n)$, where $\Pi(x,y) = y$ for $(x,y)
\in\R\times\T^{d-1}$.

We similarly define an embedded Markov chain $Z$ for the process $X^+$.
(By periodicity of $X^+$, the choice of $k$ is unimportant for the law
of $Z$, so that we drop its dependence of $k$.)
Denote by $\pi^{(k)}$ the invariant measure for $Z^{(k)}$ and by
$\pi$ the invariant measure for $Z$. We then define $\sigma$-finite
measures $\mu_+$ and $\mu^{(k)}$ on
$\R\times\T^{d-1}$ through the identities
%
%
\begin{eqnarray}
\label{e:muA}
\mu^{(k)}(B) &=& \int_{\T^{d-1}} \E_{x + ke_1} \int_0^{\phi
_1^{(k)}} \one_{B}(X(s)) \,ds\, \pi^{(k)}(dx) ,\\
\label{e:mu+A}
\mu_+(B) &=& \int_{\T^{d-1}} \E_{x + ke_1} \int_0^{\phi_1^{(k)}}
\one_{B}\bigl(X^+(s)-k\bigr) \,ds\, \pi(dx) .
\end{eqnarray}
[Here and below we make a slight abuse of notation and identify
elements $x \in\T^{d-1}$ with the element $(0,x) \in\R\times\T^{d-1}$.]
It follows from \cite{MR0133871}, Theorem 2.1, that $\mu^{(k)}$ is
invariant for the process $X$ and
$\mu_{+}$ is invariant for $X^+$. Therefore, there exist constants
$c_k > 0$ such that $\mu^{(k)} = c_k \mu$ since the invariant measure
for $X$ is unique up to normalization.
Note that by translation invariance of $X^+$, $\mu_+$ does not depend
on $k$.

Note that we can assume without any loss of generality that $A \subset
\{x \dvtx x_1 > 0\}$
[it suffices to shift it by a finite number of steps to the right in
(\ref{e:convA})]. In this case, we can rewrite
(\ref{e:muA}) as
%
%
\begin{equation}\label{e:muAprime}
\mu^{(k)}(A+k) = \int_{\T^{d-1}} \E_{x + ke_1} \int_0^{\phi
_1^{(k)}} \one_{A}\bigl(X^+(s)-k\bigr) \,ds\, \pi^{(k)}(dx) .
\end{equation}
This is because $X(t) = X^+(t)$ for $t \le\sigma_1^{(k)}$ and, if
$X(\sigma_1^{(k)}) < k$, then
\[
\int_{\sigma_1^{(k)}}^{\phi_1^{(k)}} \one_{A}\bigl(X(s)-k\bigr) \,ds = 0,
\]
whereas if $X(\sigma_1^{(k)}) > k$, then
$X(t) = X^+(t)$ for $t \le\phi_1^{(k)}$. This shows that the claim
follows if we can show that
$\|\pi- \pi^{(k)}\|_\TV\to0$ as $k \to\infty$ and there exists a
constant $c_\infty$ such that $c_k \to c_\infty$.

Let us first show that the latter is a consequence of the former.
Setting $B_k = [k,k+1]\times\T^{d-1}$,
we have $c_{k+1} / c_k = \mu^{(k)}(B_{k+1}) / \mu^{(k+1)}(B_{k+1})$.
On the other hand a straightforward trial/error argument allows one to
show that
$\E_x \int_0^{\phi_1^{(0)}} \one_{A}(X^+(s)) \,ds$
is bounded uniformly over $x \in\T^{d-1}$.
It then follows immediately
from (\ref{e:muAprime}) that there exists a constant $C$ such that
\[
\bigl|\mu^{(k)}(B_{k+1}) - \mu(B_{0})\bigr| \le C\bigl\|\pi- \pi^{(k)}\bigr\|_\TV,
\]
and similarly for $|\mu^{(k+1)}(B_{k+1}) - \mu(B_{0})|$. It follows
that provided that $\sum_{k \ge0} \|\pi- \pi^{(k)}\|_\TV< \infty$,
one does indeed have $c_k \to c_\infty$.

Denote now by $P$ the transition probabilities for $Z$ and by $P^{(k)}$
the transition probabilities for $Z^{(k)}$.
Then, we can write $P = Q R$, where $R$ is the Markov kernel from $\T
^{d-1}$ to $\{-1,1\} \times\T^{d-1}$
given by $R(x,A) = \PP_{x} (X^+(\sigma_1) \in A)$ and $Q$ is the
Markov kernel from $\{-1,1\} \times\T^{d-1}$
to $\T^{d-1}$ given by $Q(x,A) = \PP_x (X^+(\phi_0) \in A)$ for
$X_{1}(0)=0$, $\sigma_1 = \inf\{t>0\dvtx|X_{1}(t)|=1\}$ and $\phi_1 =
\inf\{t>\sigma_1\dvtx X_{1}(t)=0\}$. Since the diffusion $X^+$ is
elliptic, both
$Q$ and $R$ are strong Feller and irreducible. It follows from the
Doeblin--Doob--Khas'minskii
theorem \cite{DPZ96}, Proposition 4.1.1, that $P(x,\cdot)$ and
$P(y,\cdot)$ are mutually equivalent for any $x,y \in\T^{d-1}$.
Furthermore, it follows from the Meyer--Mokobodzki theorem
\cite{DM83,Sei02,Hair07} that the map $x \mapsto P(x,\cdot)$ is
continuous in the total variation topology. We conclude that
the map $(x,y) \mapsto\|P(x,\cdot)-P(y,\cdot)\|_\TV$ reaches
its maximum and that this
is strictly less than $2$, so that $P$ satisfies Doeblin's condition.
It follows that there exists a constant $\eta< 1$ such that $P$ has
the contraction
property
\[
\|P\nu_1 - P\nu_2\|_\TV\le\eta\|\nu_1 - \nu_2\|_\TV,
\]
for any two probability measures $\nu_1$, $\nu_2$ on $\T^{d-1}$.
Therefore, if we can find constants
$\eps_k$ such that
%
%
\begin{equation}\label{e:wantedbound}
\sup_{x \in\T^{d-1}} \bigl\|P(x,\cdot) - P^{(k)}(x,\cdot)\bigr\|_\TV\le
\eps_k ,
\end{equation}
then we have
%
%
\begin{eqnarray}\label{e:boundTVIM}
\bigl\|\pi- \pi^{(k)}\bigr\|_\TV&\le&\bigl\|P\pi- P\pi^{(k)}\bigr\|_\TV+ \bigl\|P\pi^{(k)}
- P^{(k)}\pi^{(k)}\bigr\|_\TV\nonumber\\[-8pt]\\[-8pt]
&\le&\eta\bigl\|\pi- \pi^{(k)}\bigr\|_\TV+ \eps_k ,\nonumber
\end{eqnarray}
so that $\|\pi- \pi^{(k)}\|_\TV\le\eps_k/(1-\eta)$. The problem
thus boils down to obtaining (\ref{e:wantedbound})
for an exponentially decaying sequence $\eps_k$.

It follows from the
same calculation as in Lemma \ref{lemgam} that the probability that
$X$ reaches the interface $\I_\eta$
before time $\phi_1^{(k)}$ when started on the hyperplane $\{x_1 = k\}
$ is bounded from above by $\CO(1/k)$.
This yields the ``trivial'' bound $\eps_k \le\CO(1/k)$, which
unfortunately is not even summable.
However, a more refined
analysis allows to obtain Proposition \ref{prop:convpi} below, thus
concluding the proof.
\end{pf*}
\begin{proposition}\label{prop:convpi}
There exists a constant $\rho\in(0,1)$ such that $\eps_k \le\CO
(\rho^k)$.
\end{proposition}
\begin{pf}
The intuitive idea behind the proof of Proposition \ref{prop:convpi}
is that if the process goes all the way
back to the interface then, by the time it reaches again the plane $\{
x_1 = k\}$, its hitting distribution
depends only very little on its behavior near the interface. In order
to formalize this, let us introduce
the Markov transition kernel $Q_+$ from $\T^{d-1}$ to $\T^{d-1}$
which is such that $Q_+(x,\cdot)$ is
the hitting distribution of the plane $\{1\}\times\T^{d-1}$ for the
process $X_+$ started at $(0,x)$.\vspace*{1pt}
Similarly, we denote by $Q^{\ell,k}(x,\cdot)$ the hitting
distribution of the plane $\{k\}\times\T^{d-1}$
for the process $X$ started at $(\ell,x)$.

For a fixed integer $\ell> \eta$, our aim is to show that $Q^{\ell
,k}(x,\cdot)$ gets very close to
$Q_+^{k-\ell}(x,\cdot)$. Here, we denote by $Q_+^k$ the $k$th
iteration of the Markov transition
kernel $Q_+$. With this notation at hand, define the quantities
\[
\alpha_k \equiv{\sup_{x \in\T^{d-1}}} \|Q^{\ell,k}(x,\cdot) -
Q_+^{k-\ell}(x,\cdot)\|_\TV.
\]
Note now that since, for fixed $\ell$, the probability that $X$
reaches the interface $\I_\ell$
before time $\phi_1^{(k)}$ when started on the hyperplane $\{x_1 = k\}
$ is bounded from above by $\CO(1/k)$,
we have
%
%
\begin{eqnarray}\label{e:boundepsk}
\eps_k & \le &{\sup_{x \in\T^{d-1}} }\|Q^{k-1,k}(x,\cdot) -
Q_+(x,\cdot)\|_\TV\nonumber\\[-8pt]\\[-8pt]
& \le &{{C\over k} \sup_{x \in\T^{d-1}}} \|Q^{\ell,k}(x,\cdot) -
Q^{k-\ell}_+(x,\cdot)\|_\TV\le{C\over
k}\alpha_k ,\nonumber
\end{eqnarray}
so that it suffices to obtain an exponentially decaying bound on the
$\alpha_k$'s.

We now look\vspace*{1pt} for a recursion relation on the $\alpha_k$'s which then
yields the required bound.
We have the identities $Q^{\ell,k} = Q^{k-1,k} Q^{\ell,k-1}$ and
$Q_+^{k-\ell} = Q_+ Q_+^{k-\ell-1}$.
It follows from the triangle inequality that one has the bound
%
%
\begin{eqnarray}\label{e:boundalphak}
\|Q^{\ell,k}\delta_x - Q_+^{k-\ell}\delta_x\|_\TV&\le&\|(Q^{k-1,k}
- Q_+)Q^{\ell,k-1}\delta_x\|_\TV\nonumber\\[-8pt]\\[-8pt]
&&{} + \|Q_+(Q^{\ell,k-1}\delta_x
-Q_+^{k-\ell-1}\delta_x)\|_\TV.\nonumber
\end{eqnarray}
At this stage, we note that by exactly the same reasoning as for $P$,
the kernel $Q_+$
satisfies Doeblin's condition. Therefore, there exists a constant $\bar
\eta< 1$ such that
\[
\|Q_+\nu_1 - Q_+\nu_2\|_\TV\le\bar\eta\|\nu_1 - \nu_2\|_\TV,
\]
for any two probability measures $\nu_1$, $\nu_2$. This and the
definition of $\alpha_k$
immediately implies that the second
term in (\ref{e:boundalphak}) is uniformly bounded by $\bar\eta
\alpha_{k-1}$.
On the other hand, it follows from (\ref{e:boundepsk}) that the first
term is bounded by ${C\over k}\alpha_k$,
so that
\[
\alpha_k \le{C\over k}\alpha_k + \bar\eta\alpha_{k-1} ,
\]
for some fixed constant $C$. The claim now follows at once.
\end{pf}

Finally, the last estimate that we need is the following. Denote by
$\tau$ the first hitting time of the interface
$\partial\I_\eta$ and fix an arbitrary smooth positive function
$\varphi$ that is supported
in the interval $[1,2]$. Set furthermore $\varphi_n^+(x) = n^{-2}
\varphi(n^{-1}x_1)$ and $\varphi_n^-(x) = n^{-2} \varphi(-n^{-1}x_1)$.
Then we have the following lemma.
\begin{lemma}\label{lem:convper}
With the above notation, setting $\bar\varphi= \int_1^2 \varphi
(x) \,dx$, we have
\[
\biggl|\E_x \int_0^\tau\varphi_n^\pm(X_{\pm}(t)) \,dt - {2\bar
\varphi\over D_{11}^\pm} \biggr| \to0 ,
\]
uniformly for all $x \in\{\pm n\} \times\T^{d-1}$ as $n \to\infty$.
\end{lemma}
\begin{pf}
Again, we only consider the expression for $X_+$, the one for $X_-$
follows in the same way.
It follows from standard homogenization results \cite{lions,MR2382139}
that the law of $n^{-1} X_+(n^2 t)$
converges weakly as $n \to\infty$ to the law of Brownian motion with
diffusion coefficient $D_{11}^+$. It thus follows from
\cite{MR2267655}, Corollary 8.4.2,
that the law of $n^{-1} X_+(n^2 t)$, where $X_+$ is stopped at the
first hitting time of $\I_\eta$ converges weakly as $n \to\infty$
to the law of Brownian motion
stopped when it reaches the hyperplane $\I_0$.

Denoting this limiting process by $X_+^\infty$, an explicit
calculation allows to check that
$\E_x \int_0^\tau\varphi(X_{+}^\infty(t)) \,dt = {2\bar\varphi
\over D_{11}^\pm}$ when $x_1 = 1$.
Now, for any fixed $T>0$, the map $\Phi_T\dvtx X \mapsto\int
_0^{\tau\wedge T} \varphi_n^+(X(t)) \,dt$ is continuous,
so that $\E_x \int_0^{\tau\wedge T} \varphi_n^+(X_{+}(t)) \,dt$
converges as $n \to\infty$ to
$\E_x \int_0^{\tau\wedge T} \varphi(X_{+}^\infty(t)) \,dt$. Letting
$T \to\infty$ concludes the proof.
\end{pf}

We now have all the tools that we need to show that the exit
probabilities from the interface
converge to the desired limiting values.
\begin{pf*}{Proof of Proposition \ref{probabilityconverge}}
Similarly to the proof of Proposition \ref{invariantconv} we use a
representation of the invariant measure $\mu$ in
terms of an embedded Markov chain. This time, we consider the stopping times
$\tilde{\phi}^{(k)}_0 = \inf\{t \ge0 \dvtx |X_1(t)| = \eta\}$ and then
$\tilde{\sigma}^{(k)}_n = \inf\{t \ge\tilde{\phi}^{(k)}_n \dvtx
|X_1(t)| = k\}$, $\tilde{\phi}^{(k)}_{n+1} = \inf\{t \ge\tilde
{\sigma}^{(k)}_n \dvtx |X_1(t)| = \eta\}$.
Denoting as similar to before by $\tilde{\pi}^{(k)}$ the invariant
measure of the embedded Markov chain
$\tilde{Z}^{(k)}_n = X(\tilde{\phi}^{(k)}_n)$ (which is now a Markov
chain on $\partial\I_\eta$), we set
%
%
\begin{equation} \label{e:muAnew}
\tilde{\mu}^{(k)}(B) = \int_{\partial\I_\eta} \E_{x} \int
_0^{\tilde{\phi}_1^{(k)}} \one_{B}(X(s)) \,ds\, \tilde{\pi
}^{(k)}(dx) .
\end{equation}
Again, the measures $\tilde{\mu}^{(k)}$ differ from $\mu$ purely
through a scaling factor, so that there are constants
$C_k$ such that $\tilde{\mu}^{(k)}(B) = \tilde{C}_k \mu(B)$ for
every measurable set $B$.

The idea now is to evaluate $\tilde{\mu}^{(k)} (\varphi_k^\pm)$ in
two different ways and to compare the resulting answers.
First, we note from Proposition \ref{invariantconv} that
\[
\mu^{(k)} (\varphi_k^\pm) = {C_k \over k} \bigl(q_\pm\bar\varphi
+ \CO(k^{-1}) \bigr) .
\]
On the other hand, combining Proposition \ref{converges} and
Lemma \ref{lem:convper} with
the definition (\ref{e:muAnew}), we see that
%
%
\begin{equation}\label{e:normmuk}
\mu^{(k)} (\varphi_k^\pm) = {2 p_\pm^{(k)} \bar\varphi\over
D_{11}^\pm} + o(1)
\end{equation}
as $k \to\infty$. Combining these two identities, we see that
\[
{p_+^{(k)} \over p_-^{(k)}} = {D_{11}^+ q_+ \over D_{11}^- q_-} + o(1) ,
\]
thus concluding the proof.
\end{pf*}

\section{Computation of the drift along the interface}
\label{sec:drift}

This section is devoted to the computation of the drift coefficients
$\alpha_j$ along the interface.
Denote by $\tau^n$ the first hitting time of $\partial\I_n$ by the
process $X$.
With this notation, recall that,
by (\ref{othercomponents}), we have the identity
%
%
\begin{equation}\label{e:defalpha}
\alpha_j = \lim_{n \to\infty} {1\over n} \E_x \int_0^{\tau^n}
b_{j}(X_s) \,ds ,
\end{equation}
provided that this limit exists and is independent (and uniform) over
starting points $x \in\I_\eta$.
\begin{proposition}\label{prop:convalpha}
The expression on the right-hand side in (\ref{e:defalpha}) converges
to the expression given by
(\ref{e:valuealpha}), uniformly in $x \in\I_\eta$.
\end{proposition}

In order to show this, we will use the same construction as in the
proof of Proposition \ref{probabilityconverge}.
In particular, recall the definition (\ref{e:muAnew}) of the measures
$\tilde{\mu}^{(k)}$, which are
nothing but multiples of the invariant measure $\mu$, as well as the
sequence of stopping times
$\tilde{\phi}_n^{(k)}$ and $\tilde{\sigma}_n^{(k)}$. Denote
furthermore by $\tilde{\pi}_n^{(k)}$ the invariant measure
for the process on $\partial\I_\eta$ with transition probabilities
$P(x,A)$ given by
%
%
\begin{equation}\label{e:defP}
P(x,A) \stackrel{\mathrm{def}}{=}\PP_x\bigl(X\bigl(\tilde{\phi
}_1^{(k)}\bigr) \in A | \tau^n
> \tilde{\phi}_1^{(k)}\bigr) .
\end{equation}

Our proof will proceed in two steps. First, we show that the limit
(\ref{e:defalpha}) exists and is equal to the
value (\ref{e:valuealpha}) given in the interface, provided that we
start the process $X$ in the stationary
measure $\tilde{\pi}^{(k)}_n$ and let $k \to\infty$. In the second
step, we then show by a coupling argument similar to
the proof of Proposition \ref{converges} that the expression in (\ref
{e:defalpha}) depends only weakly
on the initial condition as $n$ gets large, thus concluding the proof.

Before we proceed with this program, we perform the following
preliminary calculation.
\begin{lemma}\label{lem:normmuk}
One has the normalization
\[
\lim_{k \to\infty} k^{-2} \tilde{\mu}^{(k)}([-k,k] \times\T
^{d-1}) = 2 \biggl({p_+ \over D_{11}^+} + {p_- \over
D_{11}^-} \biggr) \stackrel{\mathit{def}}{=}\beta,
\]
where the coefficients $p_\pm$ are as in (\ref{whichsideitescapes}).
In particular, if $\mu$ is normalized as in
the \hyperref[sec1]{Introduction}, then one has $k^{-1} \tilde{\mu
}^{(k)} \approx
\beta\mu$ for large values of $k$.
\end{lemma}
\begin{pf}
We know from Proposition \ref{invariantconv} that $\mu(dx) \to\mu
_\pm(dx)$ at exponential
rate as $x_1 \to\pm\infty$, so that on large scales $\mu$ behaves
like a multiple of Lebesgue measure
on either side of the interface. Furthermore, we know from
Proposition~\ref{probabilityconverge}
that the corresponding normalization constants satisfy the relation
(\ref{whichsideitescapes}).
Combining this with the fact that $\tilde{\mu}^{(k)}$ is just a
multiple of $\mu$, the result then follows
from (\ref{e:normmuk}).
\end{pf}

Using this result, we obtain the following proposition.
\begin{proposition}\label{prop:convalphap}
The limit
\[
\alpha_j = \lim_{k \to\infty} \lim_{n \to\infty} {1\over n} \E
_{\tilde{\pi}_n^{(k)}} \int_0^{\tau^n} b_{j}(X_s) \,ds ,
\]
exists and is equal to
%
%
\begin{equation}\label{e:valuecorr}
\beta\int_{\R\times\T^{d-1}} \bigl(b_j(x) + \CL g_j(x)\bigr)
\mu(dx) ,
\end{equation}
where $g$ is the function fixed in Section \ref{sec:tight} and the
constant $\beta$ is as in~Lemma~\ref{lem:normmuk}.
\end{proposition}
\begin{remark}
Note that if $\phi$ is any smooth compactly supported function, then
the identity $\int\CL\phi(x) \mu(dx) = 0$ holds. As a consequence, the
expression (\ref{e:valuecorr}) is independent of the choice of the
compensator $g$.
\end{remark}
\begin{pf*}{Proof of Proposition \ref{prop:convalphap}}
It follows from the definition of $\tilde{\pi}_n^{(k)}$ and the
strong Markov property of $X$ that one has the
identity
%
%
\begin{eqnarray}\label{e:intbt}
&&\E_{\tilde{\pi}_n^{(k)}} \int_0^{\tau^n} \tilde b_{j}(X_s) \,ds\nonumber\\
&&\qquad=
\sum_{m \ge0} \bigl(\PP_{\tilde{\pi}_n^{(k)}}\bigl(\tilde{\phi}_1^{(k)}
< \tau^n\bigr)\bigr)^m \E_{\tilde{\pi}_n^{(k)}} \int_0^{\tilde{\phi
}_1^{(k)}\wedge\tau^n} \tilde b_{j}(X_s) \,ds \\
&&\qquad= {\E_{\tilde{\pi}_n^{(k)}} \int_0^{\tilde{\phi}_1^{(k)}\wedge
\tau^n}
\tilde b_{j}(X_s) \,ds \over\PP(\tilde{\phi}_1^{(k)} > \tau^n)
} .\nonumber
\end{eqnarray}
Note now that it follows from Lemma \ref{lemgam} that
%
%
\begin{equation}\label{e:probphi}
\PP\bigl(\tilde{\phi}_1^{(k)} > \tau^n\bigr) = {k/n + \CO(1/n)} .
\end{equation}
Since $\lim_{n\rightarrow\infty} g_{j}(X(\tau^{n}))/n=0$ and
furthermore, using the same argument as in (\ref{e:boundTVIM}), we have
$\lim_{n \to\infty} \|\tilde{\pi}_n^{(k)} - \tilde{\pi}^{(k)}\|
_\TV= 0$ for every $k > 0$, so that
%
%
\begin{eqnarray}\label{alphajinv}\qquad
\lim_{n \to\infty} {1\over n} \E_{\tilde{\pi}_n^{(k)}} \int
_0^{\tau^n} b_{j}(X_s) \,ds
& = &\lim_{n \to\infty} {1\over n} \E_{\tilde{\pi}_n^{(k)}}
\biggl[\int_0^{\tau^n} b_{j}(X_s) \,ds + g_{j}(X(\tau_n)) \biggr]
\nonumber\\
& = &\lim_{n \to\infty} {1\over n} \E_{\tilde{\pi}_n^{(k)}} \int
_0^{\tau^n} \tilde{b}_{j}(X_s) \,ds \nonumber\\
&=& \lim_{n \to\infty} {1\over k} \E_{\tilde{\pi}^{(k)}} \int
_0^{\tilde{\phi}_1^{(k)} \wedge\tau^n} \tilde{b}_{j}(X_s) \,ds
\\
&=& {1\over k} \E_{\tilde{\pi}^{(k)}} \int_0^{\tilde{\phi
}_1^{(k)}} \tilde{b}_{j}(X_s) \,ds\nonumber\\
& = &{1\over k} \int_{\R\times\T^{d-1}}
\tilde{b}_{j}(x) \tilde{\mu}^{(k)}(dx) .\nonumber
\end{eqnarray}
Here, we used (\ref{e:intbt}) and (\ref{e:probphi}) to go from the
second to the third line and we used the definition of the $\tilde{\mu}^{(k)}$
to obtain the last identity.
The claim now follows from Lemma \ref{lem:normmuk}.
\end{pf*}

We can now complete the proof.
\begin{pf*}{Proof of Proposition \ref{prop:convalpha}}
In view of Proposition \ref{prop:convalphap}, it remains to show that
\[
\lim_{n \to\infty} {1\over n} \biggl| \E_x \int_0^{\tau^n}
b(X_s) \,ds - \E_y \int_0^{\tau^n} b(X_s) \,ds \biggr| = 0 ,
\]
uniformly over $x,y \in\I_\eta$. Fix an arbitrary value of $k > \eta
$ and consider again the transition
probabilities $P$ given by (\ref{e:defP}). Since they arise as exit
probabilities for an elliptic diffusion,
we can show again by the same argument as in the proof of
Proposition \ref{invariantconv}
that $P$ satisfies the Doeblin
condition for some constant $\eta$, namely $\|P\nu_1 - P\nu_2\|_\TV
\le(1-\eta) \|\nu_1 - \nu_2\|_\TV$, uniformly
over probability measures $\nu_1$ and $\nu_2$ on $\partial\I_\eta
$. Note now that one has the identity
%
%
\begin{eqnarray}\label{e:exprx}\quad
\E_x \int_0^{\tau^n} b(X_s) \,ds &=& \sum_{m \ge0} \biggl(\prod_{0
\le\ell< m} \PP_\ell^x \bigl(\tilde{\phi}_1^{(k)} < \tau^n\bigr) \biggr)
\E_m^x \int_0^{\tilde{\phi}_1^{(k)} \wedge\tau^n} b(X_s)
\,ds\nonumber\\[-8pt]\\[-8pt]
&=& \sum_{m \ge0} \PP_x \bigl(\tilde{\phi}_m^{(k)} < \tau^n\bigr)
\E_m^x \int_0^{\tilde{\phi}_1^{(k)} \wedge\tau^n}
b(X_s) \,ds ,\nonumber
\end{eqnarray}
where we denote by $\PP_m$ (resp., $\E_m$) the probability (resp.,
expectation) for the process $X$ started
at $P^m(x,\cdot)$.

Note now that we have the identity
\[
\PP_x \bigl(\tilde{\phi}_m^{(k)} < \tau^n\bigr) = \PP_x \bigl(\tilde{\phi
}_\ell^{(k)} < \tau^n\bigr) + \PP_{P^\ell(x,\cdot)}
\bigl(\tilde{\phi}_{m-\ell}^{(k)} < \tau^n\bigr) .
\]
Also, by choosing $k$ sufficiently large (but independent of $n$), we
can ensure that there exist
constants $c,C>0$ such that
\[
1 - {C\over n} \le\PP_x \bigl(\tilde{\phi}_1^{(k)} < \tau^n\bigr) \le1 -
{c\over n} ,
\]
uniformly for $x \in\I_\eta$ and for $n$ sufficiently large.
It also follows from the contraction properties of $P$ that
\[
\bigl|\PP_m^x \bigl(\tilde{\phi}_1^{(k)} < \tau^n\bigr) - \PP_m^y \bigl(\tilde
{\phi}_1^{(k)} < \tau^n\bigr)\bigr| \le2(1-\eta)^m ,
\]
uniformly over $x,y \in\I_\eta$.

Combining these bounds, we obtain for every $\ell\le m \wedge n$ the estimate
\[
\bigl|\PP_x \bigl(\tilde{\phi}_m^{(k)} < \tau^n\bigr) - \PP_y \bigl(\tilde{\phi
}_m^{(k)} < \tau^n\bigr)\bigr|
\le{K \ell\over n} + 2(1-\eta)^\ell.
\]
In particular, there exists a constant $K$, such that we have the
uniform bound
\[
\bigl|\PP_x \bigl(\tilde{\phi}_m^{(k)} < \tau^n\bigr) - \PP_y \bigl(\tilde{\phi
}_m^{(k)} < \tau^n\bigr)\bigr|
\le{K \over\sqrt n}\wedge{Km\over n} \wedge\biggl(1-{c\over n}
\biggr)^m ,
\]
valid for every $m > 0$ and every $n$ sufficiently large. Summing over
$m$, it follows that
\[
\sum_{m \ge0}\bigl|\PP_x \bigl(\tilde{\phi}_m^{(k)} < \tau^n\bigr) - \PP_y
\bigl(\tilde{\phi}_m^{(k)} < \tau^n\bigr)\bigr|
\le K \sqrt n ,
\]
for a possibly different constant $K$.

On the other hand, it is possible to check that there exists a constant $C$
(depending on $k$) such that
\[
\biggl|\E_x \int_0^{\tilde{\phi}_1^{(k)} \wedge\tau^n} b(X_s)
\,ds \biggr| \le C ,
\]
uniformly over $x \in\I_\eta$, so that
\[
\biggl|\E_m^x \int_0^{\tilde{\phi}_1^{(k)} \wedge\tau^n} b(X_s)
\,ds - \E_m^y \int_0^{\tilde{\phi}_1^{(k)} \wedge\tau^n} b(X_s)
\,ds \biggr| \le2C (1-\eta)^m .
\]
Inserting these bounds into (\ref{e:exprx}), we obtain
\[
\biggl| \E_x \int_0^{\tau^n} b(X_s) \,ds - \E_y \int_0^{\tau^n}
b(X_s) \,ds \biggr|
\le2C\sum_{m \ge0} (1-\eta)^m + C\sqrt n ,
\]
so that the requested bound follows at once.
\end{pf*}

\subsection{Bound on the second moment}

In order to conclude the verification of the assumptions of
Theorem \ref{maintheorem},
it remains to show that the second bound holds in~(\ref
{othercomponents}). For the nonrescaled
process, we can reformulate this as
a proposition.
\begin{proposition}
For every $\bar\eta> 0$, there exists a constant $C>0$ such that the bound
\[
\E_y \|Y(\tau^n) - y\|^2 \le C n^2 ,
\]
holds for every $n \ge1$ and every initial condition $y \in\I_{\bar
\eta}$.
\end{proposition}
\begin{pf}
It follows from (\ref{e:defY}) that
%
%
\begin{equation}\label{e:boundE2}\qquad
\E_y \|Y(\tau^n) - y\|^2
\le2 \E_y \biggl\|\int_0^{\tau^n} \tilde b(X_s) \,ds \biggr\|^2 +
2 \E_y \biggl\|\int_0^{\tau^n} \tilde\sigma(X_s) \,dW(s) \biggr\|^2 .
\end{equation}
It follows from It\^o's isometry that the second term is bounded by $C
\E\tau^n$, which
in turn is bounded by $\CO(n^2)$ by a calculation virtually identical
to that of Lemma \ref{lem:boundexit}.

It remains to bound the first term, which we will do with the help of a
decomposition similar to
that used in the proof of Proposition \ref{probabilityconverge}. For
two constants $c > 0$ and $a>0$ to be
determined, we
set $\phi_0 = 0$, $\sigma_n = \inf\{t \ge\phi_{n} \dvtx |X_1(t)| =
c+a\}$ and
$\phi_n = \inf\{t \ge\sigma_{n-1} \dvtx|X_1(t)| = c\}$.
Define furthermore
\[
N = \inf\{k \ge0 \dvtx\sigma_k \ge\tau^n\} .
\]
Since $\tilde b$ is supported in a bounded strip around $\I_0$, we can
make $c$ sufficiently large so that
the first term in (\ref{e:boundE2}) is bounded by some multiple of
\begin{eqnarray*}
\E_y \Biggl(\sum_{k=0}^N (\sigma_k - \phi_k) \Biggr)^2
&\le&\sqrt{\E_y N^3 \E_y \sum_{k=0}^N (\sigma_k - \phi_k)^4}\\
&\le&\sqrt{\E_y N^3 \sum_{k=0}^\infty\E_y\bigl((\sigma_k - \phi
_k)^4 | N \ge k\bigr) \PP_y(N \ge k)} .
\end{eqnarray*}
Note now that since $\sigma_k$ is the exit time from a compact region
for an elliptic diffusion, there exists a
constant $C$ such that $\E_y((\sigma_k - \phi_k)^4 | N
\ge k) \le C$, uniformly in~$y$.
Furthermore, it follows from Lemma \ref{lemgam} that if $a$ is
sufficiently large, then
\[
\PP_y (N > 1) \le1 - {c \over n} ,
\]
for some constant $c>0$, uniformly in $y$. The strong Markov property
then immediately implies that
$\PP_y (N > k) \le(1 - {c \over n})^k$, so that $N$ is
stochastically bounded by a Poisson random variable with parameter $\CO
(n)$ and the claim follows.
\end{pf}

\section{Well-posedness of the martingale problem and characterization
of the limiting process}
\label{sec:wellposed}

The aim of this section is to show that the martingale problem
associated to the operator $\bar\CL$
as defined in Theorem \ref{maintheorem} is unique and to characterize
the corresponding
(strong) Markov process.
Our main tool is the following general result by Ethier and Kurtz \cite
{MR838085}, Theorem 4.1.
\begin{theorem} \label{333222}
Let $E$ be a separable metric space, and let $A\dvtx\mathscr{D}(A)
\to\CB_b(E)$ be linear and dissipative. Suppose there exists $\lambda
> 0$ such that
%
%
\begin{equation}\label{e:denserange}
\CC\stackrel{\mathit{def}}{=}\overline{\mathscr{R}(\lambda-A)}
=\overline{ \mathscr
{D}(A)} ,
\end{equation}
and such that $\CC$ is separating.
Let $\mu\in\mathscr{P}(E)$ and suppose $X$ is a solution of the
martingale problem for $(A,\mu)$. Then $X$ is a Markov
process corresponding to the semigroup on $\CC$ generated by the
closure of $A$, and uniqueness holds
for the martingale problem for $(A,\mu)$.
\end{theorem}

See also \cite{MR917679} for a more general result on the
well-posedness of a martingale problem with discontinuous coefficients.
This allows us to finally give the
proof of Theorem \ref{maintheorem}.
\begin{pf*}{Proof of Theorem \ref{maintheorem}}
Since we already know from the results in the previous two sections
that limit points of
$X^\eps$ solve the martingale problem associated to $\bar\CL$, it
suffices to show that this
martingale problem is well-posed and that its solutions are of the form
(\ref{e:limit}).

For this, we somehow take the reverse approach: first, we construct a
solution to (\ref{e:limit})
and we show that this is a Markov process solving the martingale
problem associated to $\bar\CL$.
We then show that this Markov process generates a strongly continuous
semigroup on $\CC_0(\R^d)$,
whose generator is the closure of $\bar\CL$ in $\CC_0$. Since $\CC
_0$ is separating and since generators of
strongly continuous semigroups are dissipative and satisfy (\ref
{e:denserange}) by the Hille--Yosida theorem,
the claim then follows.

In order to construct a solution to (\ref{e:limit}), let $M_\pm$ be
matrices satisfying $M_\pm M_\pm^T = D^\pm$
and such that
\[
M_\pm=\pmatrix{\sqrt{D^\pm_{11}} & 0 \cr v_\pm& \tilde M_\pm},
\]
for some vectors $v_\pm\in\R^{d-1}$ and some $(d-1)\times(d-1)$
matrices $\tilde M_\pm$. (This is always
possible by the QR decomposition.) We then first construct a Wiener
process $W_1$ and a process $\bar X_1$
such that
\[
d \bar X_1 = \bigl(\one_{\bar X_1 \le0} \sqrt{D_{11}^-} + \one
_{\bar X_1 > 0} \sqrt{D_{11}^+}\bigr) \,dW(t)
+ (p_+ - p_-) \,dL(t) ,
\]
where $L$ is the symmetric local time of $\bar X_1$ at the origin. This
can be achieved, for example, by
setting $\bar X_1 = g(Z)$, where
\[
g(x) =
\cases{
\sqrt{D_{11}^+}, &\quad if $x>0$,\cr
\sqrt{D_{11}^-}, &\quad otherwise,}
\]
$Z$ is a skew-Brownian motion with parameter
\[
p = {p_+ \sqrt{D_{11}^-}\over p_+ \sqrt{D_{11}^-} + p_- \sqrt
{D_{11}^+}} ,
\]
and $W$ is the martingale part of $Z$. Given such a pair $(\bar X_1,
W)$, we then
let $\tilde W$ be an independent $d-1$-dimensional Wiener process and
we define pathwise
the $\R^{d-1}$-valued process $\tilde X$ by
\begin{eqnarray*}
\tilde X(t) &=& \int_0^t (\one_{\bar X_1 \le0} \tilde M_- +
\one_{\bar X_1 > 0} \tilde M_+) \,d\tilde W(t)
+ \int_0^t (\one_{\bar X_1 \le0} v_- + \one_{\bar X_1 > 0}
v_+) \,dW(t)\\
&&{} + \tilde\alpha\int_0^t dL(t) ,
\end{eqnarray*}
where $\tilde\alpha_j = \alpha_{j+1}$. Since we know that
skew-Brownian motion enjoys
the Markov property, it follows
immediately that $\bar X_1$ is Markov, so that $\bar X = (\bar X_1,
\tilde X)$ is also a Markov process.
Applying the symmetric It\^o--Tanaka formula to $f(\bar X)$ it is
furthermore a straightforward exercise to check that
$\bar X$ does indeed solve the martingale problem for $\bar\CL$.

The corresponding Markov semigroup $\{\CP_t\}_{t\ge0}$ maps $\CC
_0(\R^d)$ into itself as a consequence
of the Feller property of skew-Brownian motion \cite{MR2280299}.
Furthermore, as a consequence of the
uniform stochastic
continuity of $\bar X$, it is strongly continuous, so that its
generator must be an extension of $\bar\CL$.
Since the range of $\bar\CL$ contains $\CC_0^\infty(\R^d)$, which
is a dense subspace of
$\CC_0(\R^d)$, the claim follows.
\end{pf*}

\section*{Acknowledgments}

The authors are grateful to Ruth Williams and Tom Kurtz for pointing
them to several articles related to this problem, as well as to the
referee for suggesting several improvements to the article. 

%

%
\printaddresses

\end{document}